\RequirePackage[l2tabu,orthodox]{nag}
\documentclass[a4paper,11pt]{amsart}
\usepackage{amsmath,amsfonts,amsthm,amssymb,amsthm,stmaryrd}
\usepackage[all]{xy}

\usepackage[T1]{fontenc}
\usepackage[sc]{mathpazo}

\usepackage[margin=1.25in]{geometry}

\usepackage{mathrsfs}

\renewcommand{\bar}[1]{\overline{#1}}
\renewcommand{\tilde}[1]{\widetilde{#1}}

\usepackage{hyperref}
\hypersetup{
  colorlinks=true,
  allcolors=blue
}

\newcommand{\sC}{\mathscr{C}}

\newcommand{\G}{\mathscr{G}}
\newcommand{\A}{\mathscr{A}}
\newcommand{\B}{\mathscr{B}}
\newcommand{\M}{\mathscr{M}}
\newcommand{\Z}{\mathscr{Z}}

\newcommand{\sP}{\mathscr{P}}

\newcommand{\g}{\mathfrak{g}}
\newcommand{\h}{\mathfrak{h}}
\newcommand{\fu}{\mathfrak{u}}

\newcommand{\C}{\mathbb{C}}
\newcommand{\R}{\mathbb{R}}
\newcommand{\bH}{\mathbf{H}}
\newcommand{\m}{\mathbf{m}}
\newcommand{\n}{\mathbf{n}}

\newtheorem{proposition}{Proposition}[section]
\newtheorem{lemma}[proposition]{Lemma}
\newtheorem{theorem}[proposition]{Theorem}
\newtheorem{corollary}[proposition]{Corollary}

\newtheorem{theoremintro}{Theorem}
\newtheorem*{theoremintro*}{Theorem}
\newtheorem*{corollaryintro*}{Corollary}

\numberwithin{equation}{section}

\DeclareMathOperator{\Ad}{Ad}
\DeclareMathOperator{\ad}{ad}

\DeclareMathOperator{\im}{im}
\DeclareMathOperator{\Aut}{Aut}
\DeclareMathOperator{\Lie}{Lie}

\DeclareMathOperator{\Hom}{Hom}

\title{The Orbit Type Stratification of the Moduli Space of Higgs Bundles}
\author{Yue Fan}
\date{\today}
\address{Yue Fan, Department of Mathematics, University of Maryland,
  College Park, MD, 20742, USA}
\email{\texttt{yuefan@umd.edu}}

\begin{document}
\begin{abstract}
  The moduli space of Higgs bundles can be constructed as a quotient
  of an infinite-dimensional space and hence admits an orbit type
  decomposition. In this paper, we show that the orbit type
  decomposition is a complex Whitney stratification such that each
  stratum is a complex symplectic submanifold and hence admits a
  complex Poisson bracket. Moreover, these Poisson brackets glue to a Poisson
  bracket on the structure sheaf of the moduli space so that the
  moduli space is a stratified complex symplectic space.
\end{abstract}
\maketitle
\tableofcontents

\section{Introduction}
Let $\sC$ be a hyperK\"ahler manifold, and $\G$ a Lie group acting on
$\sC$ and preserving the hyperK\"ahler structure. We also assume that
$\G$ admits a complexification $\G^\C$ such that the $\G$-action on
$\sC$ can be extended to a holomorphic $\G^\C$-action with respect to
some complex structure $I$ on $\sC$. Suppose there is a hyperK\"ahler
moment map $\m=(\mu,\mu_\C)$ such that $\mu$ is a moment map for the
$\G$-action with respect to the K\"ahler form induced by $I$, and
$\mu_\C$ is a complex moment map for the $\G^\C$-action with respect
to the complex symplectic form induced by the other complex structures
$J$ and $K$. Then, we may consider the hyperK\"ahler quotient
$\m^{-1}(0)/\G$. Although the quotient $\m^{-1}(0)/\G$, in general,
may be highly singular, Mayrand showed in \cite{mayrand2018local} that
if $\sC$ is finite-dimensional, and $\G$ is compact, then the
hyperK\"ahler quotient $\m^{-1}(0)/\G$ is a complex space and can be
decomposed into smooth hyperK\"ahler manifolds by orbit types such
that the decomposition is a complex Whitney stratification. The
hyperK\"ahler structure on each stratum comes from $\sC$ and is
compatible with the complex space structure. Moreover, the complex
symplectic structure on each stratum induces a complex Poisson bracket
such that these Poisson brackets glue to a complex Poisson bracket on
the structure sheaf of the complex space $\m^{-1}(0)/\G$. Finally, the
complex space $\m^{-1}(0)/\G$ is locally biholomorphic to an affine
complex symplectic quotient such that the biholomorphism is a Poisson
map and preserves the orbit type decompositions. Note that Mayrand's
result is a natural generalization of Sjamaar-Lerman
\cite{sjamaar1991stratified} from symplectic quotients to
hyperK\"ahler quotients.

The purpose of this paper is to show that Mayrand's results can be
extended to the moduli space of Higgs bundles, which is promised in
\cite{Fan2020}. More precisely, let $X$ be a closed Riemann surface
with genus $\geq2$. To parametrize Higgs bundles, we fix a smooth
Hermitian vector bundle $E$ over $X$, and let $\g_E$ be the bundle of
skew-Hermitian endomorphisms of $E$. For convenience, we may assume
that the degree of $E$ is 0. Then, by the Chern correspondence and the
fact that $\dim_\C X=1$, the space of holomorphic structures on $E$
can be identified with the space $\A$ of unitary connections on
$E$. Set $\sC=\A\times\Omega^{1,0}(\g_E^\C)$, and the configuration
space $\B$ of the Higgs bundles with underlying smooth bundle $E$ is
defined as
\begin{equation*}
  \B=\{(A,\Phi)\in\sC\colon\bar{\partial}_A\Phi=0\}
\end{equation*}
Note that the complex gauge group $\G^\C=\Aut(E)$ acts on $\B$ and
preserves the subspaces $\B^{ss}$ and $\B^{ps}$ consisting of
semistable and polystable Higgs bundles, respectively (see
\cite{wentworth2016higgs} for more details). The moduli space of Higgs
bundles is defined as the quotient $\M=\B^{ps}/\G^\C$ equipped with
the $C^\infty$-topology. To see how hyperK\"ahler geometry comes into
the picture, let us recall that the moduli space $\M$ can be realized
as a singular hyperK\"ahler quotient in the following way. Note that
$\sC$ is an infinite-dimensional affine hyperK\"ahler manifold modeled
on $\Omega^1(\g_E)\oplus\Omega^{1,0}(\g_E)$ (see
\cite[\S6]{Hitchin1987b}). The complex gauge group $\G^\C$ acts on
$\sC$ holomorphically with respect to the complex structure $I$ that
is given by the multiplication by $\sqrt{-1}$. (In this paper, we routinely
identify $\Omega^1(\g_E)$ with $\Omega^{0,1}(\g_E^\C)$.) The subgroup
$\G$ of $\G^\C$ consisting of unitary gauge transformations preserves
the hyperK\"ahler structure. The $\G$-action also admits a
hyperK\"ahler moment map as follows. Hitchin's equation
\begin{equation}
  \label{eq:HitchinEq}
  \mu(A,\Phi)=F_A+[\Phi,\Phi^*]
\end{equation}
can be regarded as a real moment map for the $\G$-action with respect
to the K\"ahler form induced by $I$. Moreover, the holomorphicity
condition $\mu_\C(A,\Phi)=\bar{\partial}_A\Phi$ can be regarded as a
complex moment map for the $\G^\C$-action with respect to the complex
symplectic form induced by the other two complex structures $J$ and
$K$. Then, the Hitchin-Kobayashi correspondence (see
\cite{Hitchin1987b} and \cite{Wilkin2006}) states that the inclusion
$\m^{-1}(0)\hookrightarrow\B^{ps}$ induces a homeomorphism
\begin{equation*}
  i\colon\m^{-1}(0)/\G\xrightarrow{\sim}\B^{ps}/\G^\C=\M
\end{equation*}
where $\m=(\mu,\mu_\C)$, and the quotient $\m^{-1}(0)/\G$ is equipped
with the $C^\infty$-topology. Therefore, we are now in Mayrand's
setting except that $\sC$, $\G$ and $\G^\C$ are
infinite-dimensional. We now define the orbit type decompositions of
$\m^{-1}(0)/\G$ and $\M$. Let $H$ be a $\G$-stabilizer at some Higgs
bundle in $\m^{-1}(0)$ and $(H)$ the conjugacy class of $H$ in
$\G$. Consider the subspace
\begin{equation*}
  \m^{-1}(0)_{(H)}=\{(A,\Phi)\in\m^{-1}(0)\colon\G_{(A,\Phi)}\in(H)\}
\end{equation*}
It is $\G$-invariant, and the orbit type decomposition of the singular
hyperK\"ahler quotient $\m^{-1}(0)/\G$ is defined as
\begin{equation*}
  \m^{-1}(0)/\G=\coprod_{(H)}\text{ components of }\m^{-1}(0)_{(H)}/\G
\end{equation*}
By abusing the notation, we generally use $\pi$ to denote the quotient
map $\B^{ps}\to\M$ or $\m^{-1}(0)\to\m^{-1}(0)/\G$. Then, we will
prove the following. It is a slight generalization of Hitchin's
construction of the moduli space of stable Higgs bundles in \cite[\S5
and \S6]{Hitchin1987b} (cf. \cite[Proposition
2.21]{stecker2015moduli}). 
\begin{theoremintro}\label{sec:introduction-pieces-hyperkahler}
  Every stratum $Q$ in the orbit type decomposition of the hyperK\"ahler
  quotient $\m^{-1}(0)/\G$ is a locally closed smooth manifold, and
  $\pi^{-1}(Q)$ is a smooth submanifold of $\sC$ such that the
  restriction $\pi\colon\pi^{-1}(Q)\to Q$ is a smooth
  submersion. Moreover, the restriction of the hyperK\"ahler structure
  from $\sC$ to $\pi^{-1}(Q)$ descends to $Q$.
\end{theoremintro}
Similarly, if $L$ is a $\G^\C$-stabilizer at some Higgs bundle in
$\B^{ps}$, and $(L)$ denotes the conjugacy class of $L$ in $\G^\C$,
then we consider the subspace
\begin{equation*}
  \B^{ps}_{(L)}=\{(A,\Phi)\in\B^{ps}\colon(\G^\C)_{(A,\Phi)}\in(L)\}
\end{equation*}
It is $\G^\C$-invariant, and the orbit type decomposition of the
moduli space $\M$ is defined as
\begin{equation*}
  \M=\coprod_{(L)}\text{ components of }\B_{(L)}^{ps}/\G^\C
\end{equation*}
In \cite{Fan2020}, it is shown that $\M$ is a normal complex
space. Then, we will prove the following.
\begin{theoremintro}\label{sec:introduction-pieces-moduli}
  Every stratum $Q$ in the orbit type decomposition of the moduli space
  $\M$ is a locally closed complex submanifold of $\M$, and
  $\pi^{-1}(Q)$ is a complex submanifold of $\sC$ with respect to the
  complex structure $I$ such that the restriction
  $\pi\colon\pi^{-1}(Q)\to Q$ is a holomorphic submersion. This
  decomposition is a complex Whitney stratification.
\end{theoremintro}
Here, by complex Whitney stratification, we mean that the orbit type
decomposition of $\M$ is a disjoint union of locally closed complex
submanifolds such that if $Q_1\cap\bar{Q_2}\neq\emptyset$ then
$Q_1\subset\bar{Q_2}$ for any strata $Q_1$ and $Q_2$ in the
decomposition. This is called the \textit{frontier
  condition}. Moreover, this decomposition is required to satisfy
Whitney conditions $A$ and $B$. Although Whitney conditions $A$ and
$B$ are conditions for submanifolds in an Euclidean space, they make
sense for complex spaces, since they are local conditions and
invariant under diffeomorphisms (see \cite[Definition 2.2, 2.5,
2.7]{mayrand2018local} for more details).

Moreover, the Hitchin-Kobayashi correspondence $i$ preserves the orbit
type decompositions in the following way.
\begin{theoremintro}\label{sec:introduction-HK-preserves-orbit-types}
  If $Q$ is a stratum in the orbit type decomposition of
  $\m^{-1}(0)/\G$, then $i(Q)$ is a stratum in the orbit type
  decomposition of $\M$, and the restriction $i\colon Q\to i(Q)$ is a
  biholomorphism with respect to the complex structure $I_Q$ on $Q$
  coming from $\sC$ and the natural complex structure on $i(Q)$.
\end{theoremintro}
Therefore, each stratum $Q$ in the orbit type decomposition of $\M$
acquires a complex symplectic structure from the corresponding stratum
in the orbit type decomposition of $\m^{-1}(0)/\G$. As a consequence,
each $Q$ admits a complex Poisson bracket. We will show that these
Poisson brackets glue to a complex Poisson bracket on the structure
sheaf of $\M$. To state the result more precisely, we recall that any
Higgs bundle $(A,\Phi)\in\m^{-1}(0)$ defines a deformation complex
$C_{\mu_\C}$, which is an elliptic complex (see
Section~\ref{sec:preliminaries}). Let $\bH^1$ denote the harmonic
space $\bH^1(C_{\mu_\C})$. In \cite{Fan2020}, it is shown that $\bH^1$
is a complex symplectic vector space, and the $\G^\C$-stabilizer
$H^\C$ at $(A,\Phi)$ acts linearly on it and preserves the complex
symplectic structure, where $H$ is the $\G$-stabilizer at
$(A,\Phi)$. Let $\nu_{0,\C}$ be the canonical complex moment map for
the $H^\C$-action on $\bH^1$. By \cite{Fan2020}, around $[A,\Phi]$,
the moduli space $\M$ is locally biholomorphic to an open neighborhood
of $[0]$ in the complex symplectic quotient
$\nu_{0,\C}^{-1}(0)\sslash H^\C$, which is an affine geometric
invariant theory (GIT) quotient. Note that
$\nu_{0,\C}^{-1}(0)\sslash H^\C$ also has an orbit type decomposition,
since every point in $\nu_{0,\C}^{-1}(0)\sslash H^\C$ has a unique
closed orbit, and this orbit has a orbit type (see
Section~\ref{sec:whitney-conditions}). By Mayrand
\cite{mayrand2018local}, the orbit type decomposition of
$\nu_{0,\C}^{-1}(0)\sslash H^\C$ is a Whitney stratification, and each
stratum is a complex symplectic submanifold and hence admits a complex
Poisson bracket. Moreover, these Poisson brackets glue to a Poisson
bracket on the structure sheaf such that the inclusion from each
stratum to $\nu_{0,\C}^{-1}(0)\sslash H^\C$ is a Poisson map. Then, we
will prove the following.
\begin{theoremintro}\label{sec:introduction-Poisson-structure}
  There is a unique complex Poisson bracket on the structure sheaf of
  $\M$ such that the inclusion $Q\hookrightarrow\M$ is a Poisson map
  for each stratum $Q$ in $\M$. Moreover, we have the following.
  \begin{enumerate}
  \item The local biholomorphism between $\M$ and
    $\nu_{0,\C}^{-1}(0)\sslash H^\C$ preserves the orbit type
    stratifications and is a Poisson map.
    
  \item Its restriction to each stratum $Q$ in $\M$ is a complex
    symplectomorphism, and hence serves as complex Darboux coordinates
    on $Q$.
  \end{enumerate}
\end{theoremintro}

Following Mayrand \cite{mayrand2018local} and Sjamaar-Lerman
\cite{sjamaar1991stratified}, a complex space is called a
\textit{stratified complex symplectic space} if it admits a complex
Whitney stratification, a complex symplectic structure on each
stratum, and a complex Poisson bracket on the structure sheaf such that the
inclusion from each stratum to the complex space is a holomorphic
Poisson map. As a consequence of the main theorems proved in this
paper, we conclude the following.
\begin{corollaryintro*}
  The moduli space $\M$ of Higgs bundles is a stratified complex
  symplectic space with the orbit type decomposition as the complex
  Whitney stratification.
\end{corollaryintro*}

To prove Theorem~\ref{sec:introduction-pieces-hyperkahler} and the
first part of Theorem~\ref{sec:introduction-pieces-moduli}, the basic
tools are local slice theorems for the $\G$-action and the
$\G^\C$-action. Since the $\G$-action is proper, its local slice
theorem is available. To obtain a local slice theorem for the
$\G^\C$-action around Higgs bundles satisfying Hitchin's equation, we
adapt Buchdahl and Schumacher's argument in \cite[Proposition
4.5]{buchdahl2020polystability}. To prove the second part of
Theorem~\ref{sec:introduction-pieces-moduli}, we simply follow
Mayrand's arguments in \cite[\S4.6, \S4.7]{mayrand2018local}. The idea
is that the Whitney conditions and the frontier condition are local
conditions and therefore can be checked on an open neighborhood of
$[0]$ in $\nu_{0,\C}^{-1}(0)\sslash H^\C$, provided that the
biholomorphism between $\M$ and a local model
$\nu_{0,\C}^{-1}(0)\sslash H^\C$ preserves the orbit type
decompositions. We will prove that this is the case. These results
will be proved in Sections~\ref{sec:orbit-type-decomp} and
\ref{sec:whitney-conditions}. Moreover, in
Section~\ref{sec:preliminaries}, we will review results in
\cite{Fan2020}

To prove Theorem~\ref{sec:introduction-HK-preserves-orbit-types}, the
major obstacle is to show that the Hitchin-Kobayashi correspondence
preserves orbit types. We will follow Sjamaar's argument in
\cite[Theorem 2.10]{sjamaar1995holomorphic}. However, this argument
crucially relies on Mostow's decomposition for complex reductive Lie
groups. Since $\G^\C$ is infinite-dimensional, we need to extend
Mostow's decomposition to $\G^\C$ in the following way.
\begin{theoremintro}[Mostow's
  decomposition]\label{sec:introduction-mostow}
  Let $H$ be a compact subgroup of $\G$ and $\h$ its Lie algebra. The
  map
  \begin{equation*}
    \h^\perp\times_H\G\to\G^\C/H^\C\qquad[s,u]\mapsto H^\C\exp(is)u
  \end{equation*}
  is a $\G$-equivariant bijection, where $\G$ acts on both sides by
  right multiplication, and $\h^\perp$ is the $L^2$-orthogonal
  complement of $\h$ in the Lie algebra $\Omega^0(\g_E)$ of $\G$.
\end{theoremintro}
It is likely that the map mentioned in
Theorem~\ref{sec:introduction-mostow} is not only a bijection but also
a diffeomorphism. That said, for the purpose of this paper, a
bijection is all we need. Once Mostow's decomposition for $\G^\C$ is
established, the rest of the proof follows easily. To prove
Theorem~\ref{sec:introduction-mostow}, we will instead prove that the
map $H^\C\times_H(\h^\perp\times\G)\to\G^\C$ is a bijection (see
Theorem~\ref{sec:mostow-decomposition-main-theorem} for a more precise
statement). To this end, following the Heinzner and Schwarz's idea in
\cite[\S9]{heinzner2007cartan}, we will realize $\h^\perp\times\G$ as
a zero set of some moment map on $\G^\C$. Therefore, we need to show
that $\G^\C$ is a weak K\"ahler manifold and that the left $H$-action
on $\G^\C$ is Hamiltonian with a suitable moment map. In
\cite{Huebschmann2013}, Huebschmann and Leicht provided a framework to
deal with this problem. Although their results are in
finite-dimensional settings, they can be carried out for $\G^\C$
without any problems. For the sake of completeness, we provide the
details in the \nameref{sec:appendix}, and the proofs are taken or
adapted from \cite{Huebschmann2013}. Then, it will be shown that every
$H^\C$-orbit in $\G^\C$ intersects $\h^\perp\times\G$, and the
intersection is a single $H$-orbit. Here, we will use the framework
laid out in Mundet I Riera's paper \cite{Riera2000}. All these results
will be proved in Section~\ref{sec:most-decomp}.

To prove Theorem~\ref{sec:introduction-Poisson-structure}, we need to
define a complex Poisson bracket on the structure sheaf of $\M$. Since
every stratum in the orbit type decomposition has a complex Poisson
bracket, and $\M$ is a disjoint union of these strata, we may
pointwise define the complex Poisson bracket of any two holomorphic
functions on $\M$. Therefore, the real question is to answer whether
the resulting function is still holomorphic. We will show that the
local biholomorphism between $\M$ and a local model
$\nu_{0,\C}^{-1}(0)\sslash H^\C$ is a Poisson map. Then,
Theorem~\ref{sec:introduction-Poisson-structure} follows from
this. Now the key observation to see that the local biholomorphism is
a Poisson map is that the Kuranishi map $\theta$ (see \cite{Fan2020}
for the construction of Kuranishi maps and Kuranishi local models)
induces the local biholomorphism and preserves the complex symplectic
structures on $\bH^1$ and $\sC$. Moreover, all the complex symplectic
structures on the strata in the orbit type decompositions of $\M$ and
$\nu_{0,\C}^{-1}(0)\sslash H^\C$ come from those on $\sC$ and $\bH^1$.

Finally, we want to say a few words on the topologies we will be using
on various spaces throughout this paper. By definition, the moduli
space $\M$ and the hyperK\"ahler quotient $\m^{-1}(0)/\G$ are equipped
with the $C^\infty$-topology. That said, in order to use the implicit
function theorem, in most of the proofs, we need to complete the
spaces $\G^\C$, $\G$ and $\sC$ with respect to the Sobolev
$L_{k+1}^2$-norm and $L_k^2$-norm. Here, $k>1$. Then, the resulting
spaces $\G_{k+1}$ and $\G^\C_{k+1}$ are Banach Lie groups acting
smoothly on the Banach affine manifold $\sC_k$. Moreover, we also need
to extend the moment map $\m$ to a moment map $\m_k$ on $\sC_k$. By
the Sobolev multiplication theorem, this is well-defined. On the other
hand, by the regularity results in \cite[Theorem
3.17]{stecker2015moduli} and \cite[Lemma 3.11, 3.12 and Corollary
3.13]{Fan2020}, the natural maps
$\m^{-1}(0)/\G\to\m_k^{-1}(0)/\G_{k+1}$ and
$\M\to\B^{ps}_k/\G^\C_{k+1}$ are homeomorphisms. Moreover, it will be
clear in the proofs of
Theorem~\ref{sec:introduction-pieces-hyperkahler} and
\ref{sec:introduction-pieces-moduli} that they preserve orbit type
decompositions. As a result, for notational convenience, we will drop
these subscripts that indicate the Sobolev completions and work with
these Sobolev completions in the proofs whenever necessary. This
should not cause any confusion. Finally, it should be noted that
Theorem~\ref{sec:introduction-mostow}, strictly speaking, should be a
result for the Banach Lie groups $\G_{k+1}$ and $\G_{k+1}^\C$.

\vspace{0.5cm}
\noindent\textbf{Acknowledgment}. This paper is part of my Ph.D. thesis. I
would like to thank my advisor, Professor Richard Wentworth, for
suggesting this problem and his generous support and guidance.

\section{Preliminaries}\label{sec:preliminaries}
In this section, we review some useful results in \cite{Fan2020}. We
start with deformation complexes. Every Higgs bundle
$(A,\Phi)\in\m^{-1}(0)$ defines a deformation complex
\begin{equation*}
  C_{\mu_\C}\colon\qquad\Omega^0(\g_E^\C)\xrightarrow{D''}\Omega^{0,1}(\g_E^\C)\oplus\Omega^{1,0}(\g_E^\C)\xrightarrow{D''}\Omega^{1,1}(\g_E^\C)
\end{equation*}
where $D''=\bar{\partial}_A+\Phi$.
\begin{proposition}[{\cite[\S1]{Simpson1992} and \cite[\S10]{Simpson1994}}]
  $C_{\mu_\C}$ is an elliptic complex and a differential graded Lie
  algebra. Moreover, the K\"ahler's identities
  \begin{equation*}
    (D'')^*=-i[*,D']\qquad (D')^*=+i[*,D'']
  \end{equation*}
  hold, where $D'=\partial_A+\Phi^*$ and $*$ is the Hodge star.
\end{proposition}
Another elliptic complex (see \cite[p.85]{Hitchin1987b}) associated
with $(A,\Phi)$ is the following.
\begin{equation*}
  C_{Hit}\colon
  \Omega^0(\g_E)\xrightarrow{d_1}\Omega^1(\g_E)\oplus\Omega^{1,0}(\g_E^\C)\xrightarrow{d_2\oplus D''}\Omega^2(\g_E)\oplus\Omega^{1,1}(\g_E^\C)
\end{equation*}
Here, $d_1(u)=(d_Au,[\Phi,u])$ and $d_2$ is the derivative of $\mu$
\eqref{eq:HitchinEq} at $(A,\Phi)$. By direct computation, we have the
following. As a consequence, throughout this paper, we will use $\bH^1$ to denote
either $\bH^1(C_{\mu_\C})$ or $\bH^1(C_{Hit})$.

\begin{proposition}
  The map $\Omega^1(\g_E)\to\Omega^{0,1}(\g_E)$ given by
  $\alpha\mapsto\alpha''$ induces an isomorphism
  $\bH^1(C_{Hit})\xrightarrow{\sim}\bH^1(C_{\mu_\C})$, where $\alpha''$
  denotes the $(0,1)$ component of $\alpha$.
\end{proposition}

Now, we review the Kuranishi local models used to construct the moduli
space $\M$ as a normal complex space (for more details, see
\cite{Fan2020}). Fix $(A,\Phi)\in\m^{-1}(0)$ with $\G$-stabilizer $H$
and consider the subspace $\tilde{\B}=(1-H)\mu_\C$, where $H$ is the
harmonic projection from $\Omega^{1,1}(\g_E^\C)$ onto
$\bH^2(C_{\mu_\C})$. Then, locally around $(A,\Phi)$, $\tilde{\B}$ is
a complex submanifold of $\sC$. Moreover, the holomorphic map
\begin{gather*}
  F\colon\Omega^{0,1}(\g_E^\C)\oplus\Omega^{1,0}(\g_E^\C)\to\Omega^{0,1}(\g_E^\C)\oplus\Omega^{1,0}(\g_E^\C)\\
  F(\alpha,\eta)=(\alpha,\eta)+(D'')^*G[\alpha'',\eta]
\end{gather*}
is $H^\C$-equivariant and restricts to
\begin{equation*}
  F\colon\tilde{\B}\cap((A,\Phi)+\ker (D'')^*)\to\bH^1
\end{equation*}
In fact, $F$ maps an open neighborhood of $(A,\Phi)$ in
$\tilde{\B}\cap((A,\Phi)+\ker(D'')^*)$ homeomorphically onto an open
ball (in $L^2$-norm) $B\subset\bH^1$ around $0$. Its inverse, viewed
as a map $\theta\colon B\to\sC$, is called a Kuranishi map. The
hyperK\"ahler structure on $\sC$ restricts to $\bH^1$, and hence
$\bH^1$ has a linear complex symplectic structure $\omega_\C$. Since
$H^\C$ acts linearly on $\bH^1$ and preserves $\omega_\C$, there is a
standard complex moment map $\nu_{0,\C}$ on $\bH^1$ such that
$\nu_{0,\C}(0)=0$. Let $\Z=B\cap\nu_{0,\C}^{-1}(0)$. It is proved that
$\theta$ maps $\Z$ homeomorphically onto an open neighborhood of
$(A,\Phi)$ in $\B^{ss}\cap((A,\Phi)+\ker (D'')^*)$. Moreover, $x\in\Z$
has a closed $K^\C$-orbit in $\bH^1$ if and only if $\theta(x)$ is a
polystable Higgs bundle, provided that $B$ is sufficiently
small. Furthermore, $B$ can be arranged so that $\Z H^\C\sslash H^\C$
is an open neighborhood of $[0]$ in $\nu_{0,\C}^{-1}(0)\sslash H^\C$
and that $\theta$ induces a biholomorphism
$\varphi\colon \Z H^\C\sslash H^\C\to\M$ onto an open neighborhood of
$[A,\Phi]$ in $\M$. We will also call $\varphi$ a Kuranishi map. More
precisely, $\varphi[x]=[r\theta(x)]$ for any $x\in\Z$, where
$r\colon\B^{ss}\to\mu^{-1}(0)$ is the retraction defined by the
Yang-Mills-Higgs flow. The local inverse of $\varphi$ is given by
$[B,\Psi]\mapsto [x]$, where $(B,\Psi)$ can be chosen such that there
are unique $g\in\G^\C$ and $x\in\Z$ such that
$(B,\Psi)=\theta(x)g$. Moreover, $g$ and $x$ depend on $(B,\Psi)$
holomorphically.

\section{Mostow's decomposition}\label{sec:most-decomp}
In this section, we will prove Mostow's decomposition for $\G^\C$,
Theorem~\ref{sec:introduction-mostow}. In fact, we will prove the
following Theorem~\ref{sec:mostow-decomposition-main-theorem}, and
Theorem~\ref{sec:introduction-mostow} follows as a corollary.

Let $H$ be a compact subgroup of $\G$ and $\h$ its Lie algebra. The
compactness of $H$ implies that $\h$ is a finite-dimensional subspace
of $\Omega^0(\g_E)$ and hence closed. Therefore, $\h$ has a
$L^2$-orthogonal complment $\h^\perp$ in $\Omega^0(\g_E)$ so that
$\Omega^0(\g_E)=\h\oplus\h^\perp$. Moreover, let $H^\C$ be the
complexification of $H$.
\begin{theorem}[cf. {\cite[Corollary
    9.5]{heinzner2007cartan}}]\label{sec:mostow-decomposition-main-theorem}
  The map
  \begin{equation*}
    H^\C\times_H(\h^\perp\times\G)\to\G^\C\qquad[h,s,u]\mapsto h\exp(is)u
  \end{equation*}
  is a bijection, where $H$ acts on
  $H^\C\times(\h^\perp\times\G)$ by
  \begin{equation*}
    h_0\cdot(h,s,u)=(hh_0^{-1},h_0sh_0^{-1},h_0u)
  \end{equation*}
\end{theorem}

To prove Theorem~\ref{sec:mostow-decomposition-main-theorem}, we adapt
the proof of \cite[Corollary 9.5]{heinzner2007cartan}. Recall that the
polar decomposition $\mathfrak{u}(n)\times U(n)\to GL_n(\C)$ induces a
polar decomposition
\begin{equation*}
  \Omega^0(\g_E)\times\G\to\G^\C\qquad (s,u)\mapsto\exp(is)u
\end{equation*}
Via the polar decomposition, the left multiplication of $H^\C$ on
$\G^\C$ induces a left $H^\C$-action on $\Omega^0(\g_E)\times\G$. In
particular, $H$ acts on $\Omega^0(\g_E)\times\G$ by
$h_0\cdot(s,u)=(h_0sh_0^{-1},h_0u)$. In the \nameref{sec:appendix}, we
will show that both $\Omega^0(\g_E)\times\G$ and $\G^\C$ are weak
K\"ahler manifolds such that the polar decomposition is an isomorphism
of K\"ahler manifolds. Moreover, the left $H$-action is Hamiltonian
with a moment map given by
\begin{equation*}
  \kappa\colon\Omega^0(\g_E)\times\G\to\h\qquad(s,u)\mapsto Ps
\end{equation*}
where $P\colon\Omega^0(\g_E)\to\h$ is the projection. Now, we
routinely identify $\Omega^0(\g_E)\times\G$ with $\G^\C$ using the
polar decomposition. Then,
Theorem~\ref{sec:mostow-decomposition-main-theorem} follows from the
following.
\begin{lemma}\ \label{sec:mostow-decomposition-lemma}
  \begin{enumerate}
  \item Every $H^\C$-orbit in $\G^\C$ intersect $\kappa^{-1}(0)$.
    
  \item $\kappa^{-1}(0)\cap H^\C g=Hg$ for every $g\in\G^\C$.
  \end{enumerate}
\end{lemma}
\begin{proof}
  Since $H$ is a compact Lie group (hence
  finite-dimensional), \cite[Lemma 5.2 and Theorem 5.4]{Riera2000}
  apply. Therefore, it suffices to show that
  \begin{equation*}
    \lim_{t\to\infty}(\kappa(\exp(its)g),s)_{L^2}>0
  \end{equation*}
  for any $s\in\h$ and $g\in\G^\C$. Using the polar
  decomposition, we may write
  \begin{equation*}
    \exp(its)g=\exp(it\eta(t))u(t)
  \end{equation*}
  for some $\eta(t)\in\Omega^0(\g_E)$ and $u(t)\in\G$. Hence,
  \begin{equation*}
    (\kappa(\exp(its)g),s)_{L^2}=(P\eta(t),s)_{L^2}=(\eta(t),s)_{L^2}
  \end{equation*}
  Since $H^\C$ acts on $\G^\C$ freely, by \cite[Lemma 2.2]{Riera2000},
  $(\eta(t),s)_{L^2}$ is a strictly increasing function of
  $t$. Therefore, it suffices to prove that if $t\gg0$,
  $(\eta(t),s)_{L^2}\geq0$. Hence, we may assume that
  $\eta(t)\neq0$ for any $t$. By the proof of \cite[Theorem
  5.12]{trautwein2015survey}, we see that
  \begin{equation*}
    \lim_{t\to\infty}\frac{\eta(t)}{\|\eta(t)\|_{L^2}}=\frac{s}{\|s\|_{L^2}}
  \end{equation*}
  in $L^2$-norm so that
  \begin{equation*}
    \lim_{t\to\infty}\biggl(\frac{\eta(t)}{\|\eta(t)\|_{L^2}},\frac{s}{\|s\|_{L^2}}\biggr)_{L^2}=1
  \end{equation*}
  Therefore, if $t\gg0$, $(\eta(t),s)_{L^2}>0$.
\end{proof}
\begin{proof}[Proof of Theorem~\ref{sec:mostow-decomposition-main-theorem}]
  Consider the map
  \begin{equation*}
    H^\C\times_H\kappa^{-1}(0)\to\G^\C\qquad [h,s,u]\mapsto h\exp(is)u
  \end{equation*}
  The surjectivity and the injectivity follow from $(1)$ and $(2)$ in
  Lemma~\ref{sec:mostow-decomposition-lemma}, respectively. Moreover, $\kappa^{-1}(0)=\h^\perp\times\G$.
\end{proof}

As a corollary of Mostow's decomposition,
Theorem~\ref{sec:introduction-mostow}, we obtain the following that
will be used often in this paper.
\begin{corollary}\label{sec:most-decomp-1-conjugate}
  Let $H$ and $K$ be compact subgroups of $\G$. Then, $H^\C$ and
  $K^\C$ are conjugate in $\G^\C$ if and only if $H$ and $K$ are
  conjugate in $\G$.
\end{corollary}
\begin{proof}
  This follows from Mostow's decomposition
  (Theorem~\ref{sec:introduction-mostow}) and the first paragraph in
  the proof of \cite[Theorem 2.10]{sjamaar1995holomorphic}. Note that
  all we need is the fact that the map in
  Theorem~\ref{sec:introduction-mostow} is a $\G$-equivariant
  bijection.
\end{proof}

\section{The orbit type
  decompositions}\label{sec:orbit-type-decomp}

\subsection{Orbit types in the hyperK\"ahler quotient}

In this section, we will prove
Theorem~\ref{sec:introduction-pieces-hyperkahler}. 

\begin{proof}[Proof of
  Theorem~\ref{sec:introduction-pieces-hyperkahler}]
  Fix $[A,\Phi]\in Q$ such that $(A,\Phi)$ is of class
  $C^\infty$. Hence, gauge transformations in its $\G$-stabilizer $H$
  are of class $C^\infty$. By definition, $Q$ is a component of
  $\m^{-1}(0)_{(H)}/\G$. Since the $\G$-action is proper, a standard
  argument (e.g. \cite[Proposition 4.4.5]{friedman1998gauge}) shows
  that there is an $H$-invariant open neighborhood $S$ of $(A,\Phi)$
  in $(A,\Phi)+\ker d_1^*$ such that the natural map
  $f\colon S\times_H\G\to\sC$ is a $\G$-equivariant diffeomorphism
  onto an open neighborhood of $(A,\Phi)$, where $d_1$ is defined in
  the complex $C_{Hit}$ (see
  Section~\ref{sec:preliminaries}). Therefore, the restriction
  \begin{equation*}
    f\colon ((\m^{-1}(0)\cap S)\times_H\G)_{(H)}\to\m^{-1}(0)_{(H)}
  \end{equation*}
  is a $\G$-equivariant homeomorphism onto an open neighborhood of
  $(A,\Phi)$ in $\m^{-1}(0)_{(H)}$. Since $Q$ is open in
  $\m^{-1}(0)_{(H)}/\G$, $\pi^{-1}(Q)$ is open in $\m^{-1}(0)_{(H)}$
  and contains $ (A,\Phi)$. By shrinking $S$, we may further assume
  that $f$ takes values in $\pi^{-1}(Q)$. We claim that
  \begin{equation*}
    ((\m^{-1}(0)\cap S)\times_H\G)_{(H)}=(\m^{-1}(0)\cap
    S^H)\times_H\G=(\m^{-1}(0)\cap S^H)\times(\G/H)
  \end{equation*}
  where $S^H$ consists of elements in $S$ that are fixed by $H$. The
  second equality is obvious. To show the first one, let $[B,\Psi,g]$
  be a point in $(\m^{-1}(0)\cap S)\times_H\G$ with $\G$-stabilizer
  conjugate to $H$ in $\G$. As a consequence,
  \begin{equation*}
    \G_{[B,\Psi,1]}=g\G_{[B,\Psi,g]}g^{-1}\in(H)
  \end{equation*}
  Since $S$ is a local slice for the $\G$-action on $\sC$,
  $\G_{[B,\Psi,1]}\subset H$. Since $\G_{[B,\Psi,1]}$ and $H$ have the
  same dimension and the same number of components, we see that
  $\G_{[B,\Psi,1]}=H$. Hence, $H$ fixes $(B,\Psi)$, and the claim
  follows. Therefore, the map $\pi^{-1}(Q)\to Q$ can be locally
  identified with the projection
  \begin{equation*}
    (\m^{-1}(0)\cap
    S^H)\times(\G/H)\to\m^{-1}(0)\cap S^H.
  \end{equation*}
  Moreover, since $H$ is compact, the quotient map
  $\m^{-1}(0)\cap S\to(\m^{-1}(0)\cap S)/H$ is closed. Since $S^H$ is
  closed in $S$, we conclude that $(\m^{-1}(0)\cap S^H)/H$ is closed
  in $(\m^{-1}(0)\cap S)/H$. Since
  $\m^{-1}(0)\cap S^H=(\m^{-1}(0)\cap S^H)/H$ is homeomorphic to an
  open neighborhood of $[A,\Phi]$ in $Q$, $Q$ is a locally closed
  subset of $\m^{-1}(0)/\G$. Then, we prove that $\m^{-1}(0)\cap S^H$
  is a submanifold of $S^H$. As a consequence, $\pi^{-1}(Q)$ is a
  submanifold of $\sC$, $Q$ is a smooth manifold, and
  $\pi\colon\pi^{-1}(Q)\to Q$ is a smooth submersion. 

  To show that $\m^{-1}(0)\cap S^H$ is a submanifold of $S^H$, we adapt
  the proof of \cite[Theorem 2.24]{stecker2015moduli}. Let $\mu_i$ be
  a component of the hyperK\"ahler moment map $\m$. We first show that
  the restriction $\mu_i|_{S^H}$ has a constant finite corank so that
  $S^H\cap\mu_i^{-1}(0)$ is a submanifold of $S^H$. After that, we
  show that $S^H\cap\m^{-1}(0)=\cap_{i=1}^3S^H\cap\mu_i^{-1}(0)$ is a
  submanifold of $S^H$. Fix $(B,\Psi)\in S^H$. Note that
  $T_{(B,\Psi)}S^H=(\ker d_1^*)^H$. Consider the sequence
  \begin{equation*}
    \Omega^0(\g_E)\xrightarrow{d_1}T_{(B,\Psi)}\sC\xrightarrow{d\mu_i}\Omega^2(\g_E)
  \end{equation*}
  where $d\mu_i$ is the derivative of $\mu_i$ at $(B,\Psi)$. Note that
  $H$ acts on each term by conjugation, and both $d_1$ and $d\mu_i$
  are $H$-equivariant. Since the complex $C_{Hit}$ is elliptic, the
  symbol of $d\mu_i$ is surjective so that the Hodge decomposition
  \begin{equation*}
    \Omega^2(\g_E)=\ker (d\mu_i)^*\oplus\im d\mu_i
  \end{equation*}
  holds, where $(d\mu_i)^*$ is the $L^2$-formal adjoint of $d\mu_i$. We
  claim that
  \begin{equation*}
    d\mu_i\Bigl((T_{(B,\Psi)}\sC)^H\Bigr)=(\im d\mu_i)^H
  \end{equation*}
  Since $d\mu_i$ is $H$-equivariant, the inclusion ``$\subset$'' is
  obvious. Conversely, suppose $y=d\mu_i(x)$ is fixed by $H$ for some
  $x\in T_{(B,\Psi)}\sC$. Since $H$ is compact, $\int_H(x\cdot h)dh$
  is well-defined and fixed by $H$. Therefore, the inclusion
  ``$\supset$'' follows from
  \begin{equation*}
    y=\int_H(y\cdot h)dh=\int_Hd\mu_i(x\cdot
    h)dh=d\mu_i\biggl(\int_H(x\cdot h)dh\biggr)
  \end{equation*}
  Then, the $H$-equivariance of $d\mu_i$ implies that
  \begin{equation*}
    \Omega^2(\g_E)^H=(\ker (d\mu_i)^*)^H\oplus(\im d\mu_i)^H=(\ker
    (d\mu_i)^*)^H\oplus d\mu_i\Bigl((T_{(B,\Psi)}\sC)^H\Bigr)
  \end{equation*}
  Moreover, the formula $(d\mu_i)^*=-I_id_1*$ implies that
  \begin{equation*}
    \dim\ker (d\mu_i)^*=\dim\ker d_1=\dim\G_{(B,\Psi)}
  \end{equation*}
  Since $S$ is a local slice for the $\G$-action on $\sC$,
  $\G_{(B,\Psi)}=H$. Finally, since
  \begin{equation*}
    (T_{(B,\Psi)}\sC)^H=(\im d_1)^H\oplus(\ker d_1^*)^H=(\im d_1)^H\oplus(T_{(B,\Psi)}S^H)
  \end{equation*}
  and $d\mu_id_1=0$, we conclude that
  $\mu_i|_{S^H}\colon S^H\to\Omega^2(\g_E)^H$ has a constant finite
  corank so that $\mu_i^{-1}(0)\cap S^H$ is a smooth submanifold of
  $S^H$.

  Now, we show that $\cap_{i=1}^3S^H\cap\mu_i^{-1}(0)$ is a
  submanifold of $S^H$. Let $\Delta\colon S^H\to(S^H)^3$ be the
  diagonal map. If we can show that $\Delta$ is transversal to
  \begin{equation*}
    W=(\mu_1^{-1}(0)\cap S^H)\times (\mu_2^{-1}(0)\cap S^H)\times
    (\mu_3^{-1}(0)\cap S^H)\subset(S^H)^3
  \end{equation*}
  then $\Delta^{-1}(W)=\cap_{i=1}^3S^H\cap\mu^{-1}(0)$ is a smooth
  submanifold of $S^H$. So, we fix $(B,\Psi)\in\Delta^{-1}(W)$. Then,
  we have
  \begin{equation*}
    T_{\Delta(B,\Psi)}W=\bigoplus_{i=1}^3\Bigl(\ker d\mu_i\cap(\ker d_1^*)^H\Bigr)
  \end{equation*}
  and
  \begin{equation*}
    \Delta_*T_{(B,\Psi)}S^H=\{(v,v,v)\colon v\in(\ker d_1^*)^H\}
  \end{equation*}
  Note that $\ker d\mu_i=(I_i\im d_1)^\perp$. Therefore, if
  $u_i\in(\ker d_1^*)^H$ $(i=1,2,3)$, then we may write
  $u_i=u_i'+u_i''$ for $u_i'\in I_i\im d_1$ and
  $u_i''\in(I_i\im d_1)^\perp$. Since $\mu_i(B,\Psi)=0$, it is not
  hard to check that $u_i''\in(\ker d_1^*)^H$. Moreover, since
  $I_i\im d_1$ are orthogonal to each other, we may further write
  \begin{equation*}
    u_i=(u_i'-\sum_{j\neq i}u_j'')+u_1''+u_2''+u_3''
  \end{equation*}
  Therefore, $T_{\Delta(B,\Psi)}W$ and $\Delta_*T_{(B,\Psi)}S^H$
  generate $(T_{(B,\Psi)}S^H)^3$, and we are done.

  By construction, we see that $\m^{-1}(0)\cap S^H$ is a smooth
  manifold with tangent space
  \begin{equation*}
    \ker d\m\cap(\ker d_1^*)^H=(\ker d_2\cap\ker d_1^*)^H=\bH^1(C_{Hit})^H=(\bH^1)^H
  \end{equation*}
  at $[A,\Phi]$. Since the $L^2$-metric on $\sC$ is preserved by the
  $\G$-action, it descends to a metric on $Q$. By the proof of
  \cite[Theorem 6.7]{Hitchin1987b}, we see that $I$, $J$ and $K$
  restrict to $\bH^1$. Since they are preserved by the $H$-action on
  $\bH^1$, they further restrict to $(\bH^1)^H$. Moreover, since they
  are preserved by the $\G$-action, they, together with the
  $L^2$-metric on $(\bH^1)^H$, define an almost hyperK\"ahler structure
  on $Q$.

  Let $\Omega_I$, $\Omega_J$ and $\Omega_K$ be the K\"ahler forms on
  $\sC$ associated with complex structures $I$, $J$ and $K$,
  respectively. By the proof of \cite[Theorem 6.7]{Hitchin1987b}, we
  see that if $v\in\im d_1\oplus\bH^1$ then $\Omega_i(v,\xi)=0$ for
  any $i\in\{I,J,K\}$ and any vector $\xi$ tangent to the
  $\G$-orbit. This also holds for
  $v\in\im d_1\oplus(\bH^1)^H=T_{(A,\Phi)}\pi^{-1}(Q)$. Therefore,
  there are unique K\"ahler forms $\omega_i$, $i\in\{I,J,K\}$, on $Q$
  such that $\pi^*\omega_i=\Omega_i|_{\pi^{-1}(Q)}$ for
  $i\in\{I,J,K\}$. Therefore, each $\omega_i$ is closed. Then, the
  integrability of complex structures on $Q$ follows from \cite[Lemma
  6.8]{Hitchin1987b}.

  Finally, by the elliptic regularity, it is easy to see that elements
  in $\ker d_1^*\cap\m^{-1}(0)$ are of class $C^\infty$. Since
  $S^H\cap\m^{-1}(0)\subset\ker d_1^*\cap\m^{-1}(0)$, our heuristic
  use of infinite-dimensional manifolds can be justified by working
  with Sobolev completions.
\end{proof}

\subsection{A local slice theorem}\label{sec:local-slice-theorem}
Now we study the strata in the orbit type decomposition of
$\M$. Therefore, we need a local slice theorem for the $\G^\C$-action
on $\sC$.

\begin{theorem}\label{sec:local-slice-theorem-local-slice-thm}
  Let $(A,\Phi)$ be a Higgs bundle in $\m^{-1}(0)$ with
  $\G$-stabilizer $H$. Then, there exists an open neighborhood $O$ of
  $(A,\Phi)$ in $(A,\Phi)+\ker (D'')^*$ such that the natural map
  $OH^\C\times_{H^\C}\G^\C\to\sC$ is a biholomorphism onto an open
  neighborhood of $(A,\Phi)$.
\end{theorem}
\begin{proof}
  Note that the $\G^\C$-stabilizer of $(A,\Phi)$ is $H^\C$ and acts on
  $(D'')^*$. Consider the natural map
  \begin{equation*}
    f\colon((A,\Phi)+\ker (D'')^*)\times_{H^\C}\G^\C\to\sC
  \end{equation*}
  Its derivative at $[A,\Phi,1]$ is given by
  \begin{equation*}
    \ker
    (D'')^*\oplus\bH^0(C_{\mu_\C})^\perp\to\Omega^{0,1}(\g_E^\C)\oplus\Omega^{1,0}(\g_E^\C)\qquad
    (x,u)\mapsto x+D''u
  \end{equation*}
  and hence an isomorphism, since $T_{(A,\Phi)}\sC=\ker (D'')^*\oplus\im
  D''$. Therefore, there are open neighborhoods $O\times N$ of
  $(A,\Phi,1)$ and $W$ of $(A,\Phi)$ such that $f\colon\pi(O\times
  N)\to W$ is a biholomorphism, where $\pi$ is the quotient map. Then,
  we consider the restriction
  \begin{equation*}
    f\colon OH^\C\times_{H^\C}\G^\C\to\sC
  \end{equation*}
  Since $OH^\C\G^\C=W\G^\C$, its image is $W\G^\C$ which is an open
  neighborhood of $(A,\Phi)$ in $\sC$. Since
  $\G^\C=\cup_{g\in\G^\C}Ng$ and $f$ is $\G^\C$-equivariant, $f$ is a
  local biholomorphism. Therefore, it remains to show that $f$ is
  injective provided that $O$ is small enough. We will follow the
  proof of \cite[Proposition 4.5]{buchdahl2020polystability}. Suppose
  \begin{equation*}
    ((A,\Phi)+(\alpha_1,\eta_1))g=(A,\Phi)+(\alpha_2,\eta_2)
  \end{equation*}
  for some $g\in\G^\C$ and $(D'')^*(\alpha_i,\eta_i)=0$. Equivalently,
  \begin{equation*}
    D''g+(\alpha_1''g-g\alpha_2'',\eta_1g-g\eta_2)=0
  \end{equation*}
  where $\alpha_i''$ is the $(0,1)$-component of $\alpha_i$. We show
  that if each $\|(\alpha_i,\eta_i)\|_{L^2_k}$ is small enough, then
  $g\in H^\C$. Write $g=g_0+g_1$ for some $g_0\in\bH^0(C_{\mu_\C})$
  and $g_1\in\bH^0(C_{\mu_\C})^\perp$. The idea is to show that
  $D''g_1=0$ so that $g=g_0\in H^\C$. Applying $(D'')^*$, we obtain
  \begin{equation*}
    (D'')^*D''g+(D'')^*(\alpha_1''g-g\alpha_2'',\eta_1g-g\eta_2)=0
  \end{equation*}
  We first claim that $D'g_0=0$, where $D'=\partial_A+\Phi^*$. In fact,
  using K\"ahler's identity, $D'^*=+i[*,D'']$, we have
  \begin{equation*}
    \|D'g_0\|_{L^2}^2=(D'^*D'g_0,g_0)_{L^2}=i(D''D'g_0,g_0)_{L^2}=-i(D'D''g_0,g_0)_{L^2}=0
  \end{equation*}
  Here, we have used the fact that $D''D'+D'D''=0$, since
  $\mu(A,\Phi)=F_A+[\Phi,\Phi^*]=0$. Then, using the K\"ahler's
  identity, $(D'')^*=-i[*,D']$, we see that $D'(\alpha_i,\eta_i)=0$ for
  each $i$, and
  \begin{equation*}
    (D'')^*(\alpha_1''g-g\alpha_2'',\eta_1g-g\eta_2)=(D'')^*(\alpha_1''g_1-g_1\alpha_2'',\eta_1g_1-g_1\eta_2)
  \end{equation*}
  As a consequence,
  \begin{align*}
    \|D''g_1\|_{L^2}^2
    &=-((\alpha_1''g_1-g_1\alpha_2'',\eta_1g_1-g_1\eta_2),D''g_1)_{L^2}\\
    &\leq\|(\alpha_1''g_1-g_1\alpha_2'',\eta_1g_1-g_1\eta_2)\|_{L^2}\|D''g_1\|_{L^2}\\
    &\leq(\|\alpha_1''\|_{C^0}+\|\alpha_2''\|_{C^0}+\|\eta_1\|_{C^0}+\|\eta_2\|_{C^0})\|g_1\|_{L^2}\|D''g_1\|_{L^2}\\
    &\leq C(\|\alpha_1''\|_{L^2_k}+\|\alpha_2''\|_{L^2_k}+\|\eta_1\|_{L^2_k}+\|\eta_2\|_{L^2_k})\|g_1\|_{L^2}\|D''g_1\|_{L^2}
  \end{align*}
  where we have used the Sobolev embedding $L_k^2\hookrightarrow
  C^0$. Moreover, since $g_1\in\bH^0(C_{\mu_\C})^\perp$,
  \begin{equation*}
    \|g_1\|_{L^2}\leq\|g_1\|_{L^2_1}=\|(D'')^*GD''g_1\|_{L^2_1}\leq C\|D''g_1\|_{L^2}
  \end{equation*}
  Therefore,
  \begin{equation*}
    \|D''g_1\|_{L^2}^2\leq C(\|\alpha_1''\|_{L^2_k}+\|\alpha_2''\|_{L^2_k}+\|\eta_1\|_{L^2_k}+\|\eta_2\|_{L^2_k})\|D''g_1\|_{L^2}^2
  \end{equation*}
  Since the isomorphism $\Omega^1(\g_E)\to\Omega^{1,0}(\g_E^\C)$ given
  by $\alpha\mapsto\alpha''$ is a homeomorphism in the
  $L_k^2$-topology, we conclude that that if
  $\|(\alpha_i,\eta_i)\|_{L^2_k}$ is small enough for every $i$, then
  $D''g_1=0$.
\end{proof}

As a corollary of
Proposition~\ref{sec:local-slice-theorem-local-slice-thm}, every
Kuranishi map $\theta\colon B\to\sC$ (see
Section~\ref{sec:preliminaries}) preserves stabilizers in the
following sense.
\begin{proposition}\label{sec:local-slice-theorem-stabs-preservation}
  If $B$ is sufficiently small, then $(H^\C)_x=(\G^\C)_{\theta(x)}$
  for every $x\in \Z$.
\end{proposition}
\begin{proof}
  By \cite[Proposition 3.4]{Fan2020}, if $x_1=x_2g$ for some
  $x_1,x_2\in B$ and $g\in H^\C$ then $\theta(x_1)=\theta(x_2)g$. This
  proves the inclusion ``$\subset$''. To prove the inclusion
  ``$\supset$'', we shrink $B$ so that $\theta(\Z)\subset O$ where $O$
  is obtained in
  Proposition~\ref{sec:local-slice-theorem-local-slice-thm}. As a
  consequence, if $\theta(x)g=\theta(x)$ for some $g\in\G^\C$ then
  $g\in H^\C$. Since $F$ is $H^\C$-equivariant and $F(\theta(x))=x$
  for every $x\in B$, we conclude that $xg=x$.
\end{proof}

\subsection{Orbit types in the moduli
  space}\label{sec:orbit-type-decomp-moduli}

Now we are able to prove
Theorem~\ref{sec:introduction-pieces-moduli}. Before giving the proof,
we first show that there is a one-to-one correspondence between the
conjugacy classes appearing in the orbit type decompositions of $\M$
and $\m^{-1}(0)/\G$.

\begin{proposition}\label{sec:orbit-types-moduli-L=Hc}
  Every conjugacy class $(L)$ appearing in the orbit type
  decomposition of $\M$ is equal to a conjugacy class $(H^\C)$ for
  some $\G$-stabilizer $H$ at some Higgs bundle in $\m^{-1}(0)$.
\end{proposition}
\begin{proof}
  Let $(A,\Phi)$ be a polystable Higgs bundle whose $\G^\C$-stabilizer
  is conjugate to $L$ in $\G^\C$. By the Hitchin-Kobayashi
  correspondence, $\mu((A,\Phi)g)=0$ for some $g\in\G^\C$. Therefore,
  \begin{equation*}
    (\G_{(A,\Phi)g})^\C=(\G^\C)_{(A,\Phi)g}=g^{-1}(\G^\C)_{(A,\Phi)}g
  \end{equation*}
  Let $H=\G_{(A,\Phi)g}$ and the $\G^\C$-stabilizer of $(A,\Phi)$ is
  conjugate to $H^\C$ in $\G^\C$.
\end{proof}

\begin{proof}[Proof of Theorem~\ref{sec:introduction-pieces-moduli}]
  Fix $[A,\Phi]\in Q$ such that $(A,\Phi)\in\m^{-1}(0)$ is of class
  $C^\infty$. Therefore, gauge transformations in the $\G$-stabilizer
  $H$ at $(A,\Phi)$ are of class $C^\infty$. Since $H^\C$ is the
  $\G^\C$-stabilizer at $(A,\Phi)$, $Q$ is a component of
  $\B^{ps}_{(H^\C)}/\G^\C$. By
  Theorem~\ref{sec:local-slice-theorem-local-slice-thm}, there is an
  open neighborhood $O$ of $(A,\Phi)$ in $(A,\Phi)+\ker (D'')^*$ such
  that the natural map $OH^\C\times_{H^\C}\G^\C\to\sC$ is a
  diffeomorphism onto an open neighborhood of $(A,\Phi)$. The
  $\G^\C$-equivariance implies that
  \begin{equation*}
    (OH^\C\times_{H^\C}\G^\C)_{(H^\C)}\to\sC_{(H^\C)}
  \end{equation*}
  Since $H^\C$ has finitely many components and is finite-dimensional,
  following the proof of
  Theorem~\ref{sec:introduction-pieces-hyperkahler}, we see that
  \begin{equation*}
    (OH^\C\times_{H^\C}\G^\C)_{(H^\C)}=O^{H^\C}\times(\G^\C/H^\C)
  \end{equation*}
  As a consequence, the natural map
  \begin{equation*}
    (O^{H^\C}\cap\B^{ss})\times(\G^\C/H^\C)\to\B^{ss}_{(H^\C)}
  \end{equation*}
  is a diffeomorphism onto an open image. On the other hand, the
  Kuranishi map $\theta\colon\Z\to O\cap\B^{ss}$ is a homeomorphism if
  $B$ is sufficiently small. Since $\theta$ preserves stabilizers
  (Proposition~\ref{sec:local-slice-theorem-stabs-preservation}), we
  see that $\theta\colon\Z^{H^\C}\to O^{H^\C}\cap\B^{ss}$ is a
  homeomorphism. Since the $H^\C$-action on $\bH^1$ is holomorphic
  with respect to $I$, $(\bH^1)^{H^\C}$ is a complex symplectic
  subspace of $\bH^1$ so that $\bH^1=F\oplus(\bH^1)^{H^\C}$ where $F$
  is the $\omega_\C$-complement of $(\bH^1)^{H^\C}$. As a consequence,
  \begin{equation*}
    \nu_{0,\C}^{-1}(0)=(\nu_{0,\C}|_F)^{-1}(0)\times(\bH^1)^{H^\C}
  \end{equation*}
  so that $\Z^{H^\C}$ is an open subset of
  $\nu_{0,\C}^{-1}(0)^{H^\C}=(\bH^1)^{H^\C}$. Since every point in
  $(\bH^1)^{H^\C}$ has closed $H^\C$-orbits,
  $\theta(\Z^{H^\C})\subset O^{H^\C}\cap\B^{ps}$. Moreover, since $Q$
  is open in $\B^{ps}_{(H^\C)}/\G^\C$, $\pi^{-1}(Q)$ is open in
  $\B^{ps}_{(H^\C)}$. Therefore, if $O$ and $B$ are sufficiently
  small, we obtain a well-defined map
  \begin{equation*}
    f\colon\Z^{H^\C}\times(\G^\C/H^\C)\to\pi^{-1}(Q)\qquad (x,[g])\mapsto\theta(x)g
  \end{equation*}
  which is a homeomorphism onto its open image. This already shows
  that $\pi^{-1}(Q)$ is a complex submanifold of $\sC$. Moreover, $f$
  induces a well-defined map $\Z^{H^\C}\to Q$ given by
  $[x]\mapsto[\theta(x)]$. It is exactly the restriction of the local
  chart (see Section~\ref{sec:preliminaries})
  $\varphi\colon\Z H^\C\sslash H^\C\to\M$, since
  \begin{equation*}
    \nu_{0,\C}^{-1}(0)\sslash H^\C=(\nu_{0,\C}|_F)^{-1}(0)\sslash H^\C\times(\bH^1)^{H^\C}
  \end{equation*}
  and $\theta(x)$ is polystable for every $x\in\Z^{H^\C}$. Therefore,
  we see that $Q$ is a complex submanifold of $\M$. Moreover, the
  quotient map $\pi\colon\pi^{-1}(Q)\to Q$ can be locally identified
  with the projection
  \begin{equation*}
    \Z^{H^\C}\times(\G^\C/H^\C)\to\Z^{H^\C}
  \end{equation*}
  which is clearly a holomorphic submersion. Finally, by elliptic
  regularity, elements in $\bH^1$ are of class $C^\infty$. Since
  $\Z^{H^\C}\subset\bH^1$, our heuristic use of infinite-dimensional
  manifolds can be justified by working with Sobolev completions.
\end{proof}

\section{Complex Whitney stratification}\label{sec:whitney-conditions}
To show that the orbit type decomposition of $\M$ is a complex Whitney
stratification, we follow Mayrand's arguments in \cite[\S4.6 and
\S4.7]{mayrand2018local}. The idea is that the problem can be reduced
to a local model $\nu_{0,\C}^{-1}(0)\sslash H^\C$ near $[0]$, once we
show that the Kuranishi map $\varphi\colon\tilde{U}\to U$ preserves
the orbit type decompositions, where $\tilde{U}=\Z H^\C\sslash H^\C$
and $U=\varphi(\tilde{U})$ (see Section~\ref{sec:preliminaries}). To
clarify different possible partitions on
$\nu_{0,\C}^{-1}(0)\sslash H^\C$, we adopt the following notation. By
\cite{heinzner1994reduction}, the natural map
$\nu_{0,\C}^{-1}(0)^{ps}\hookrightarrow\nu_{0,\C}^{-1}(0)$ induces a
bijection
\begin{equation*}
  \nu_{0,\C}^{-1}(0)^{ps}/H^C\xrightarrow{\sim}\nu_{0,\C}^{-1}(0)\sslash H^\C
\end{equation*}
where $\nu_{0,\C}^{-1}(0)^{ps}$ is the subspace of
$\nu_{0,\C}^{-1}(0)$ consisting of polystable points, or equivalently
points whose $H^\C$-stabilizers are closed in $\bH^1$. Let $L$ be a
$H^\C$-stabilizer at some point in $\nu_{0,\C}^{-1}(0)$ and
$(L)_{H^\C}$ the conjugacy class of $L$ in $H^\C$. Then, we may define
\begin{equation*}
  \nu_{0,\C}^{-1}(0)^{ps}_{(L)_{H^\C}}=\{x\in\nu_{0,\C}^{-1}(0)^{ps}\colon (H^\C)_x\in(L)_{H^\C}\}
\end{equation*}
As a consequence, $\nu_{0,\C}^{-1}(0)\sslash H^\C$ has a partition
\begin{equation*}
  \tilde{\sP}_{H^\C}=\Bigl\{\nu_{0,\C}^{-1}(0)^{ps}_{(L)_{H^\C}}/H^\C\colon
  L=(H^\C)_x\text{ for some }x\in\nu_{0,\C}^{-1}(0)\Bigr\}
\end{equation*}
Here, we have identified $\nu_{0,\C}^{-1}(0)^{ps}_{(L)_{H^\C}}$ with
its image in $\nu_{0,\C}^{-1}(0)\sslash H^\C$. The orbit type
decomposition of $\nu_{0,\C}^{-1}(0)\sslash H^\C$ is defined as the
refinement $\tilde{\sP}_{H^\C}^\circ$ of $\tilde{\sP}_{H^\C}$ into
connected components. If $(L)_{\G^\C}$ is the conjugacy class of
$L\subset H^\C$ in $\G^\C$, then we may similarly define the partition
\begin{equation*}
  \tilde{\sP}_{\G^\C}=\Bigl\{\nu_{0,\C}^{-1}(0)^{ps}_{(L)_{\G^\C}}/H^\C\colon
  L=(H^\C)_x\text{ for some }x\in\nu_{0,\C}^{-1}(0)\Bigr\}
\end{equation*}
Finally, note that $\nu_{0,\C}^{-1}(0)\sslash H^\C$ can be realized as
a hyperK\"ahler quotient as follows. Since $H$ acts linearly on $\bH^1$
and preserves the K\"ahler form $\omega_I$, there is a moment map
$\nu_0$ associated with the K\"ahler form $\omega_I$ such that
$\nu_0(0)=0$. Then, $\n_0=(\nu_0,\nu_{0,\C})$ is a hyperK\"ahler moment
map for the $H$-action. By \cite{heinzner1994reduction}, the inclusion
$\n_0^{-1}(0)\hookrightarrow\nu_{0,\C}^{-1}(0)$ induces a
homeomorphism
\begin{equation*}
  \n_0^{-1}(0)/H\xrightarrow{\sim}\nu_{0,\C}^{-1}(0)\sslash H^\C
\end{equation*}
Then, the same proof of Proposition~\ref{sec:orbit-types-moduli-L=Hc}
shows that every conjugacy class $(L)$ appearing in
$\tilde{\sP}_{H^\C}$ or $\tilde{\sP}_{\G^\C}$ is equal to a conjugacy
class $(H_1^\C)$ for some $H$-stabilizer $H_1$ at some point in
$\n_0^{-1}(0)$. Then, the relation between $\tilde{\sP}_{H^\C}$ and
$\tilde{\sP}_{\G^\C}$ is stated below.
\begin{lemma}[{\cite[Lemma 4.2]{mayrand2018local}}]\label{sec:whitney-conditions-HC=GC-orbit-types}
  $\tilde{\sP}_{H^\C}^\circ=\tilde{\sP}_{\G^\C}^\circ$.
\end{lemma}
\begin{proof}
  This follows from the same proof of \cite[Lemma
  4.2]{mayrand2018local}. The only difference is that Mowstow's
  decomposition in that proof must be replaced by
  Theorem~\ref{sec:introduction-mostow}.
\end{proof}
Now, let $\sP$ be the partition of $\M$ defined as
\begin{equation*}
  \sP=\{\B^{ps}_{(L)_{\G^\C}}/\G^\C\colon
  L=(\G^\C)_{(A,\Phi)}\text{ for some }(A,\Phi)\in\B^{ps}\}
\end{equation*}
Note that the orbit type decomposition of $\M$ is simply
$\sP^\circ$. Then, we have the following.
\begin{proposition}\label{sec:whitney-conditions-phi-preserve-pieces}
  The Kuranishi map
  $\varphi\colon\bigl(\tilde{U},(\tilde{\sP}^\circ_{H^\C}|_{\tilde{U}})^\circ\bigr)\to
  \bigl(U,(\sP^\circ|_U)^\circ\bigr)$ is an isomorphism of partitioned
  spaces.
\end{proposition}
\begin{proof}
  We first record a simple fact without a proof.
  \begin{lemma}\label{sec:whitney-conditions-refinement-open}
    Let $X$ be a space and $\sP$ a partition of $X$. If $U$ is an open
    subset of $X$, then $(\sP^\circ|_U)^\circ=(\sP|_U)^\circ$.
  \end{lemma}
  Then, by Lemma~\ref{sec:whitney-conditions-HC=GC-orbit-types} and
  \ref{sec:whitney-conditions-refinement-open},
  \begin{equation*}
    (\tilde{\sP}^\circ_{H^\C}|_{\tilde{U}})^\circ=(\tilde{\sP}^\circ_{\G^\C}|_{\tilde{U}})^\circ=(\tilde{\sP}_{\G^\C}|_{\tilde{U}})^\circ
  \end{equation*}
  and $(\sP^\circ|_U)^\circ=(\sP|_U)^\circ$. Then, it suffices to show
  that
  \begin{equation*}
    \varphi\colon(\tilde{U},\tilde{\sP}_{\G^\C}|_{\tilde{U}})\to(U,\sP|_U)
  \end{equation*}
  is an isomorphism of partitioned spaces. If $[x]\in\tilde{U}$ is
  such that the $H^\C$-orbit of $x$ is closed in $\bH^1$, then
  $\varphi[x]=[\theta(x)]$. The rest follows from
  $(\G^\C)_{\theta(x)}=(H^\C)_x$
  (Proposition~\ref{sec:local-slice-theorem-stabs-preservation}).
\end{proof}
\begin{theorem}
  The orbit type decomposition of $\M$ is a complex Whitney
  stratification.
\end{theorem}
\begin{proof}
  We first show that $\sP^\circ$ satisfies the frontier condition. In
  other words, we need to show that if $Q_1,Q_2\in\sP^\circ$ and
  $Q_1\cap\bar{Q_2}\neq\emptyset$, then $Q_1\subset\bar{Q_2}$. Fix
  $[A,\Phi]\in\M$ such that $\mu(A,\Phi)=0$ and $\G_{(A,\Phi)}=H$. Let
  $\varphi\colon\tilde{U}\to U$ be the Kuranishi map. Let $Q$ be the
  component of $\B^{ps}_{(H^\C)}/\G^\C$ containing $[A,\Phi]$. If
  $[x]\in\varphi^{-1}(Q\cap U)$ such that its $H^\C$-orbit is closed
  in $\bH^1$, then $\varphi[x]=[\theta(x)]\in Q\cap U$. By
  Proposition~\ref{sec:local-slice-theorem-stabs-preservation},
  $(\G^\C)_{\theta(x)}=(H^\C)_x$. Since $[\theta(x)]\in Q$, we
  conclude that $(H^\C)_x=H^\C$. From the proof of
  Theorem~\ref{sec:introduction-pieces-moduli} in
  Section~\ref{sec:orbit-type-decomp-moduli}, we see that
  \begin{equation*}
    \varphi^{-1}(Q\cap U)=\Z^{H^\C}=B\cap(\bH^1)^{H^\C}
  \end{equation*}
  This shows that $Q\cap U$ is connected so that
  $Q\cap U\in(\sP^\circ|_U)^\circ$. By \cite[Lemma
  4.7]{mayrand2018local}, it suffices to show that
  $(\sP^\circ|_U)^\circ $ is conical at $Q\cap U$ (see
  \cite[p.18]{mayrand2018local} for the definition). Since
  $B\cap(\bH^1)^{H^\C}\in(\tilde{\sP}^\circ_{(H^\C)})^\circ$,
  Proposition~\ref{sec:whitney-conditions-phi-preserve-pieces} implies
  that it suffices to show that
  $(\tilde{\sP}^\circ_{(H^\C)}|_{\tilde{U}})^\circ$ is conical at
  $B\cap(\bH^1)^{H^\C}$. Moreover, by
  Lemma~\ref{sec:whitney-conditions-refinement-open}, it suffices to
  show that $(\tilde{\sP}_{(H^\C)}|_{\tilde{U}})^\circ$ is conical at
  $B\cap(\bH^1)^{H^\C}$. This follows from the proof of
  \cite[Proposition 4.8]{mayrand2018local}.

  Now we show that $\sP^\circ$ satisfies the Whitney conditions at
  every point of $\M$. Fix $[A,\Phi]\in\M$ and let
  $\varphi\colon\tilde{U}\to U$ be a Kuranishi map such that
  $[A,\Phi]\in U$. Then, it suffices to check that
  $(\sP^\circ|_U)^\circ$ satisfies the Whitney conditions at
  $[A,\Phi]$. By
  Proposition~\ref{sec:whitney-conditions-phi-preserve-pieces}, it
  suffices to check that
  $(\tilde{\sP}_{H^\C}^\circ|_{\tilde{U}})^\circ=(\tilde{\sP}_{H^\C}|_{\tilde{U}})^\circ$
  satisfies the Whitney conditions at $[0]$. This follows from
  \cite[Proposition 4.12]{mayrand2018local}.
\end{proof}

\section{The Hitchin-Kobayashi correspondence}

Finally, we show that the Hitchin-Kobayashi correspondence preserves
the orbit type decompositions, Theorem~\ref{sec:introduction-HK-preserves-orbit-types} (cf. \cite[Proposition
4.6]{mayrand2018local}).
\begin{proof}[Proof of Theorem~\ref{sec:introduction-HK-preserves-orbit-types}]
  Suppose $Q$ is a component of $\m^{-1}(0)_{(H)}/\G$ for some
  $\G$-stabilizer at a Higgs bundle in $\m^{-1}(0)$. We first show
  that the restriction
  \begin{equation*}
    i\colon\m^{-1}(0)_{(H)}/\G\to\B^{ps}_{(H^\C)}/\G^\C
  \end{equation*}
  is a bijection and hence a homeomorphism. As a consequence, $i(Q)$
  is a stratum in the orbit type decomposition of $\M$. The injectivity
  is obvious. To show the surjectivity, let
  $[A,\Phi]\in\B^{ps}_{(H^\C)}/\G^\C$. We may further assume that
  $(\G^\C)_{(A,\Phi)}=H^\C$. The Hitchin-Kobayashi correspondence
  provides some $g\in\G^\C$ such that $(A,\Phi)g\in\m^{-1}(0)$. Hence,
  \begin{equation*}
    (\G_{(A,\Phi)g})^\C=(\G^\C)_{(A,\Phi)g}=g^{-1}(\G^\C)_{(A,\Phi)}g=g^{-1}H^\C g
  \end{equation*}
  Then, Corollary~\ref{sec:most-decomp-1-conjugate} implies that
  $\G_{(A,\Phi)g}$ is conjugate to $H$ in $\G$.

  Then, we show that the restriction $i\colon Q\to i(Q)$ is
  holomorphic. Consequently, since $Q$ and $i(Q)$ are smooth, $i|_Q$
  is a biholomorphism. By the proofs of
  Theorem~\ref{sec:introduction-pieces-hyperkahler} and
  \ref{sec:introduction-pieces-moduli}, we see that $i|_Q$ can be
  locally identified with a map $\m^{-1}(0)\cap
  S^H\to(\bH^1)^{H^\C}$. More precisely, this map is given by
  $(B,\Psi)\mapsto x$ where $x$ is determined by the equation
  $(B,\Psi)=\theta(x)g$ for a unique $g\in\G^\C$. It is holomorphic,
  since $x$ and $g$ depend on $(B,\Psi)$ holomorphically, which can be
  seen by Proposition~\ref{sec:local-slice-theorem-local-slice-thm}
  and the holomorphicity of $\theta\colon B\to\sC$.
\end{proof}

\section{Poisson structure}\label{sec:poisson-structure}
Let $Q$ be a stratum in the orbit type stratification of $\M$. As a
consequence, $i^{-1}(Q)$ is a stratum in the orbit type stratification
of $\m^{-1}(0)/\G$, where $i$ is the Hitchin-Kobayashi
correspondence. We have shown in
Theorem~\ref{sec:introduction-pieces-hyperkahler} that $i^{-1}(Q)$ is
a hyperK\"ahler manifold and hence has a complex symplectic form
$\omega_\C=\omega_J+\sqrt{-1}\omega_K$. Using the Hitchin-Kobayashi
correspondence $i$, we may transport $\omega_\C$ to $Q$ so that $Q$ is
also a complex symplectic manifold. We will still use $\omega_\C$ to
denote the resulting complex symplectic form on $Q$. Alternatively,
$\omega_\C$ can be defined as follows. Let $\pi\colon\B^{ps}\to\M$ be
the quotient map. It is shown in
Theorem~\ref{sec:introduction-pieces-moduli} that $\pi^{-1}(Q)$ is a
complex submanifold of $\sC$, and $\pi\colon\pi^{-1}(Q)\to Q$ is a
holomorphic submersion. Then, it follows that
$\pi^*\omega_\C=\Omega_\C|_{\pi^{-1}(Q)}$, where
$\Omega_\C=\Omega_J+\sqrt{-1}\Omega_K$ is the complex symplectic form on
$\sC$. This can be seen by
Theorem~\ref{sec:introduction-pieces-hyperkahler} and the definition
of $i$. Then, a complex Poisson bracket can be defined on the structure sheaf
of $\M$ as follows. Let $U$ be an open subset of $\M$ and
$f,g\colon U\to\C$ holomorphic functions. Let $Q$ be a stratum in the
orbit type stratification of $\M$. Therefore, the restrictions
$f|_{U\cap Q}$ and $g|_{U\cap Q}$ are holomorphic so that the Poisson
bracket $\{f|_{U\cap Q},g|_{U\cap Q}\}_Q$ is well-defined using the
complex symplectic form $\omega_\C$ on $Q$.  Consequently, there is a
unique function $\{f,g\}\colon U\to\C$ such that
\begin{equation*}
  \{f,g\}|_{U\cap Q}=\{f|_{U\cap Q},g|_{U\cap Q}\}_Q
\end{equation*}
for every stratum $Q$. Then, it remains to show that
$\{f,g\}\colon U\to\C$ is holomorphic. Since the complex space $\M$ is
constructed by gluing Kuranishi local models, $\{f,g\}$ is holomorphic
if and only if its pullback along any Kuranishi map is holomorphic. On
the other hand, from Section~\ref{sec:whitney-conditions}, we see that
$\nu_{0,\C}^{-1}(0)\sslash H^\C$ can be realized as a singular
hyperK\"ahler quotient. Hence, the structure sheaf of
$\nu_{0,\C}^{-1}(0)\sslash H^\C$ has a Poisson structure by
\cite[Theorem 1.4]{mayrand2018local}. Then, the holomorphicity of the
Poisson bracket on $\M$ follows from the following result
(cf. \cite[Proposition 4.18]{mayrand2018local}).
\begin{theorem}[=Theorem~\ref{sec:introduction-Poisson-structure}]\label{sec:poisson-structure-Kura-Poisson-map}
  Every Kuranishi map $\varphi\colon\tilde{U}\to U$ is a Poisson
  map. In other words, if $f,g\colon U\to\C$ are holomorphic
  functions, then $\varphi^*\{f,g\}=\{\varphi^*f,\varphi^*g\}$.
\end{theorem}
Before giving the proof, we need the following lemmas.
\begin{lemma}\label{sec:poisson-structure-theta-symplectomorphisms}
  The Kuranishi map $\theta\colon B\to\sC$ preserves the complex
  symplectic forms.
\end{lemma}
\begin{proof}
  By construction of $\theta$ (see
  Section~\ref{sec:local-slice-theorem}), it suffices to show that the map
  \begin{gather*}
    F\colon\tilde{B}\cap((A,\Phi)+\ker (D'')^*)\to\bH^1\\
    F(\alpha'',\eta)=(\alpha'',\eta)+\frac{1}{2}(D'')^*G[\alpha'',\eta;\alpha'',\eta]
  \end{gather*}
  preserves the restrictions of K\"ahler forms $\Omega_J$ and
  $\Omega_K$. For notational convenience, let
  $M=\tilde{\B}\cap((A,\Phi)+\ker (D'')^*)$. If
  $(B,\Psi)=(A,\Phi)+(\alpha''_0,\eta_0)\in M$, then
  \begin{equation*}
    d_{(B,\Psi)}F(\alpha,\eta)=(\alpha,\eta)+(D'')^*G[\alpha''_0,\eta_0;\alpha,\eta]\qquad
    (\alpha,\eta)\in T_{(B,\Psi)}M
  \end{equation*}
  Then, we need to show that
  \begin{equation*}
    \Omega_i(d_{(B,\Psi)}F(\alpha''_1,\eta_1),d_{(B,\Psi)}F(\alpha''_2,\eta_2))=\Omega_i(\alpha''_1,\eta_1;\alpha''_2,\eta_2)
  \end{equation*}
  for any $i\in\{J,K\}$, $(B,\Psi)\in M$, and
  $(\alpha''_j,\eta_j)\in T_{(B,\Psi)}M$. This amounts to show the following
  \begin{enumerate}
  \item
    $\Omega_i(\alpha''_1,\eta_1;(D'')^*G[\alpha_0'',\eta_0;\alpha''_2,\eta_2])=0$
    
  \item $\Omega_i((D'')^*G[\alpha_0,\eta_0,\alpha''_1,\eta_1];(D'')^*G[\alpha_0'',\eta_0;\alpha''_2,\eta_2)=0$
  \end{enumerate}
  for any $i\in\{J,K\}$ and $(\alpha''_j,\eta_j)\in\ker (D'')^*$
  $(j=0,1,2)$. To show $(1)$, we compute
  \begin{equation*}
    \Omega_J(\alpha''_1,\eta_1;(D'')^*G[\alpha_0'',\eta_0;\alpha''_2,\eta_2])
    =g(D''J(\alpha''_1,\eta_1);G[\alpha_0'',\eta_0;\alpha''_2,\eta_2])
  \end{equation*}
  where $g$ is the $L^2$-metric. Moreover,
  \begin{equation*}
    D''J(\alpha''_1,\eta_1)=D''(i\eta^*,-i\alpha''^*)=-i\bar{\partial}_A\alpha''^*+i[\Phi,\eta^*]=(iD'(\alpha''_1,\eta_1))^*=0
  \end{equation*}
  where the last equality follows from the K\"ahler's identity,
  $(D'')^*=-i[*,D']$ and the assumption that
  $(D'')^*(\alpha''_1,\eta_1)=0$. Similarly,
  \begin{equation*}
    \Omega_K(\alpha''_1,\eta_1;(D'')^*G[\alpha_0'',\eta_0;\alpha''_2,\eta_2])
    =g(D''K(\alpha''_1,\eta_1);G[\alpha_0'',\eta_0;\alpha''_2,\eta_2])
  \end{equation*}
  and
  \begin{equation*}
    D''K(\alpha''_1,\eta_1)=D''(-\eta^*,\alpha''^*)=\bar{\partial}_A\alpha''^*-[\Phi,\eta^*]=D'(\alpha_1'',\eta_1)^*=0
  \end{equation*}
  Finally, the same argument shows $(2)$.
\end{proof}
Now, let $\tilde{C}$ be a connected component of
$\tilde{Q}\cap\tilde{U}$, where $\tilde{Q}$ is a stratum in the orbit
type stratification of $\nu_{0,\C}^{-1}(0)\sslash H^\C$. In other
words, $\tilde{C}\in(\sP^\circ_{H^\C}|_{\tilde{U}})^\circ$. By
Proposition~\ref{sec:whitney-conditions-phi-preserve-pieces}, there is
some connected component $C$ of $Q\cap U$ for some stratum $Q$ in the
orbit type stratification of $\M$ such that the restriction
$\varphi\colon\tilde{C}\to C$ is a biholomorphism. Let $\pi$ denote
the projections $\B^{ps}\to\M$ and
$\nu_{0,\C}^{-1}(0)^{ps}\to\nu_{0,\C}^{-1}(0)\sslash H^\C$. By
Theorem~\ref{sec:introduction-pieces-moduli} and \cite[Lemma
4.14]{mayrand2018local}, $\pi^{-1}(C)$ and $\pi^{-1}(\tilde{C})$ are
complex submanifolds of $\sC$ and $\bH^1$, respectively. Moreover, the
following diagram commutes
\begin{equation*}
  \xymatrix
  {
    \pi^{-1}(\tilde{C})\cap B\ar[d]^-{\pi}\ar[r]^-{\theta}&\pi^{-1}(C)\ar[d]^-{\pi}\\
    \tilde{C}\ar[r]^-{\varphi} &C
  }
\end{equation*}
Now, by Lemma~\ref{sec:poisson-structure-theta-symplectomorphisms},
$\theta$ preserves the restrictions of the complex symplectic forms
$\Omega_\C$ on $\sC$ and $\omega_\C$ on $\bH^1$. By
Theorem~\ref{sec:introduction-pieces-moduli} and \cite[Lemma
4.14]{mayrand2018local} again, we see that these restrictions of
complex symplectic forms descend to $\tilde{C}$ and $C$. As a
consequence, we obtain the following.
\begin{lemma}
  The Kuranishi map $\varphi\colon\tilde{C}\to C$ is a complex
  symplectomorphism. In particular, it preserves the complex Poisson brackets.
\end{lemma}
\begin{proof}[Proof of Theorem~\ref{sec:poisson-structure-Kura-Poisson-map}]
  Let $f,g\colon U\to\C$ be holomorphic functions. Then, we compute
  \begin{align*}
    (\varphi^*\{f,g\})|_{\tilde{C}}
    &=(\varphi|_{\tilde{C}})^*(\{f,g\}|_C)\\
    &=(\varphi|_{\tilde{C}})^*(\{f|_C,g|_C\}_Q)\\
    &=\{(\varphi|_{\tilde{C}})^*(f|_C),(\varphi|_{\tilde{C}})^*(g|_C)\}_{\tilde{Q}}\\
    &=\{(\varphi^*f)|_{\tilde{C}},(\varphi^*g)|_{\tilde{C}}\}_{\tilde{Q}}\\
    &=\{\varphi^*f,\varphi^*g\}|_{\tilde{C}}
  \end{align*}
  for any connected component $\tilde{C}$ of $\tilde{Q}\cap\tilde{U}$
  for some stratum $\tilde{Q}$ in the orbit type stratification of
  $\nu_{0,\C}^{-1}(0)\sslash H^\C$. This completes the proof.
\end{proof}

\section{Appendix}\label{sec:appendix}
In this section, we will prove that $\G^\C$ is
a weak K\"ahler manifold such that the left $\G$-action on $\G^\C$ is
Hamiltonian with a moment map $\kappa\colon\G^\C\to\Omega^0(\g_E)$
given by $\kappa(\exp(is)u)=s$. Most of the proofs are taken or
adapted from \cite{Huebschmann2013}.

We first describe a weak symplectic form on
$\Omega^0(\G_E)\times\G$. Define the 1-form $\tau$ on
$\Omega^0(\g_E)\times\G$ by
\begin{equation*}
  \tau_{(s,u)}(z,w) =(s,wu^{-1})_{L^2}\qquad
  (z,w)\in\Omega^0(\g_E)\oplus T_u\G
\end{equation*}
where $(\cdot,\cdot)_{L^2}$ is the $L^2$-metric on
$\Omega^0(\g_E)$. Note that if $\G$ were finite-dimensional, then
$\tau$ would be exactly the tautological 1-form on the cotangent
bundle $T^*\G=\Lie(\G)\times\G$. There are left and right actions of
$\G$ on $\Omega^0(\g_E)\times\G$ given by
\begin{equation*}
  u_0\cdot(s,u)=(u_0su_0^{-1},u_0u)\text{ and }(s,u)\cdot u_0=(s,uu_0)
\end{equation*}
By direct computation, we see that both the left and right
$\G$-actions preserve $\tau$ and hence the 2-form $\omega:=-d\tau$. In
this section, a map is said to be left (resp. right) $\G$-equivariant
if it is $\G$-equivariant with respect to the left (resp. right)
$\G$-action, and $\G$-equivariant if it is $\G$-equivariant with
respect to both the left and right $\G$-actions.
\begin{proposition}
  If the 2-form $\omega$ is non-degenerate, then the projection
  $\kappa\colon\Omega^0(\g_E)\times\G\to\Omega^0(\g_E)$ onto the first
  factor is a moment map for the left $\G$-action on
  $\bigl(\Omega^0(\g_E)\times\G,\omega\bigr)$.
\end{proposition}
\begin{proof}
  To verify that $\kappa$ is a moment map for the left $\G$-action on
  $\Omega^0(\g_E)\times\G$, fix $\xi\in\Omega^0(\g_E)$ and let $\xi^*$
  denote the vector field generated by the left $\G$-action on
  $\Omega^0(\g_E)\times\G$. Since the left $\G$-action preserves
  $\omega$, the Lie derivative of $\omega$ along $\xi^*$
  vanishes. Therefore, $-i_{\xi^*}d\tau=di_{\xi^*}\tau$. Moreover,
  \begin{equation*}
    \xi^*_{(s,u)}=\frac{d}{dt}\biggr|_{t=0}(\Ad(e^{t\xi})s,e^{t\xi}u)=([\xi,s],\xi u)
  \end{equation*}
  and hence $\tau(\xi^*)(s,u)=(s,\xi)_{L^2}$. Finally, it is easy to
  verify that $\kappa$ is left $\G$-equivariant.
\end{proof}

Now, we describe a complex structure $J$ on $\Omega^0(\g_E)\times\G$
and later verify that $\omega$ is compatible with $J$ and positive so
that $\omega$ is a K\"ahler form on $\Omega^0(\g_E)\times\G$. Let
$\psi\colon\Omega^0(\g_E)\times\G\to\G^\C$ be the polar decomposition
given by $\psi(s,u)=\exp(is)u$. It is clear that $\psi$ is
$\G$-equivariant. Then, there is a unique complex structure $J$ on
$\Omega^0(\g_E)\times\G$ such that $\psi$ is a biholomorphism. To see
the relation between the symplectic form $\omega$ and the complex
structure $J$, we also view $P=\Omega^0(\g_E)\times\G$ as a principal
$\G$-bundle over $\Omega^0(\g_E)$, and show that $P$ has a connection
induced by the complex structure $J$.
\begin{proposition}\label{sec:appendix-connection}
  Every tangent vector of $P$ at a point $(s,u)\in P$ can be uniquely
  written as $\xi^\#_{(s,u)}+J\eta^\#_{(s,u)}$ for some
  $\xi,\eta\in\Omega^0(\g_E)$, where $\xi^\#_{(s,u)}$ is the tangent
  vector of $P$ at $(s,u)$ generated by the right $\G$-action.
\end{proposition}
\begin{proof}
  Note that any tangent vector of $\G^\C$ at $\psi(s,u)$ can be
  uniquely written as $Z^\#_{\psi(s,u)}$ for some
  $Z\in\Omega^0(\g_E^\C)$, where $Z^\#_{\psi(s,u)}$ is the tangent
  vector on $\G^\C$ generated by the right translations. Then, write
  $Z=\xi+i\eta$ for some $\xi,\eta\in\Omega^0(\g_E)$. Since the right
  $\G$-action and $\psi$ are holomorphic, and $\psi$ is
  $\G$-equivariant, we obtain
  \begin{align*}
    Z^\#_{\psi(s,u)}
    &=(\xi+i\eta)^\#_{\psi(s,u)}\\
    &=\xi^\#_{\psi(s,u)}+i\eta^\#_{\psi(s,u)}\\
    &=d_{(s,u)}\psi(\xi^\#_{(s,u)}+J\eta^\#_{(s,u)})
  \end{align*}
  where $i$ in the second equality also denotes the complex structure
  on $\G^\C$. The rest follows from the fact that the derivative
  $d_{(s,u)}\psi$ is an isomorphism.
\end{proof}
By Proposition~\ref{sec:appendix-connection}, we are able to define a
$\Omega^0(\g_E)$-valued 1-form $\gamma$ by
\begin{equation*}
  \gamma(\xi^\#_{(s,u)}+J\eta^\#_{(s,u)})=\xi
\end{equation*}
Another $\Omega^0(\g_E)$-valued 1-form $\chi$ on $P$ is given by
\begin{equation*}
  \chi(\xi^\#_{(s,u)}+J\eta^\#_{(s,u)})=\eta
\end{equation*}
It is clear that $\chi=-\gamma J$. Since $\psi$ is $\G$-equivariant,
the right $\G$-action on $P$ is holomorphic. Therefore, it is easy to
verify that both $\gamma$ and $\chi$ are $\G$-equivariant in the
sense that $R(u_0)^*\gamma u_0=\Ad(u_0^{-1})\gamma$ and
$R(u_0)^*\chi=\Ad(u_0^{-1})\chi$, where $R(u_0)$ is the right
$\G$-action on $P$ given by $u_0$. The following are some useful
formulas.
\begin{proposition}\label{sec:appendix-formula-lambda-chi}
  For any $(s,e)\in P$ and
  $(z,w)\in T_{(s,e)}P=\Omega^0(\g_E)\oplus\Omega^0(\g_E)$, the
  formulas for $\chi$ and $\gamma$ are given by
  \begin{align*}
    \gamma_{(s,e)}(z,w)&=\frac{1-\cos\ad s}{\ad s}z+w\\
    \chi_{(s,e)}(z,w)&=\frac{\sin\ad s}{\ad s}z
  \end{align*}
\end{proposition}
\begin{proof}
  The derivative $d_{(s,e)}\psi$ is given by
  \begin{align*}
    d_{(s,e)}\psi(z,w)
    &=\frac{d}{dt}\biggr|_{t=0}\exp(is+itz)\exp(tw)\\
    &=\frac{d}{dt}\biggr|_{t=0}\exp(is+itz)+\frac{d}{dt}\biggr|_{t=0}\exp(is)\exp(tw)
  \end{align*}
  Moreover, by the formula for the derivative of the exponential map
  (e.g. \cite[Theorem 1.5.3]{duistermaat2012lie})
  \begin{align*}
    \exp(is)^{-1}\frac{d}{dt}\biggr|_{t=0}\exp(is+itz)
    &=\frac{1-\exp(-\ad(is))}{\ad(is)}(iz)\\
    &=\Bigl(\frac{1-\cos\ad s}{\ad s}+i\frac{\sin\ad s}{\ad s}\Bigr)z
  \end{align*}
  As a consequence,
  \begin{equation*}
    \psi(s,e)^{-1}d_{(s,e)}\psi(z,w)
    =\frac{1-\cos\ad s}{\ad s}z+w+i\frac{\sin\ad s}{\ad s}z
  \end{equation*}
  The rest follows from the proof of
  Proposition~\ref{sec:appendix-connection}.
\end{proof}

Consider a right $\G$-equivariant map $\bar{\kappa}$ given by
$\bar{\kappa}(s,e)=\kappa(s,e)$. In other words, $\bar{\kappa}(s,u)=u^{-1}su$.
\begin{proposition}\ \label{sec:appendix-existence-Psi}
  \begin{enumerate}
  \item $\tau=(\bar{\kappa},\gamma)_{L^2}$.
  \item If $f\colon P\to\R$ is a function given by
    $f(s,u)=\frac{1}{2}\|s\|_{L^2}$, then
    $(\bar{\kappa},\chi)_{L^2}=df$.
    
  \item There is a unique right $\G$-equivariant
    $\Hom(\Omega^0(\g_E),\Omega^0(\g_E))$-valued 1-form $\Psi$ on $P$
    such that $d_\gamma\bar{\kappa}_{(s,u)}=\Psi_{(s,u)}\chi_{(s,u)}$
    for any $(s,u)\in P$, where
    $d_\gamma\bar{\kappa}$ is the
    covariant derivative of $\bar{\kappa}$. More explicitly,
    \begin{align*}
      d_\gamma\bar{\kappa}_{(s,e)}(z,w)&=\cos\ad s(z)\\
      \Psi_{(s,e)}&=\cos\ad s\frac{\ad s}{\sin\ad s}
    \end{align*}
    for any $(z,w)\in T_{(s,e)}P$.
  \end{enumerate}
\end{proposition}
Before giving the proof, we claim that $\frac{\sin\ad s}{\ad s}$ is
invertible so that $\frac{\ad s}{\sin\ad s}$ simply means its
inverse. In fact, if $\xi\in\Omega^0(\g_E)$, then, by definition,
$\chi_{(s,e)}(J\xi^\#_{(s,e)})=\xi$. Moreover, if
$J\xi^\#_{(s,e)}=(z,w)$ for some $(z,w)\in T_{(s,e)}P$, then
\begin{equation*}
  \xi=\chi_{(s,e)}(J\xi^\#_{(s,e)})=\chi_{(s,e)}(z,w)=\chi_{(s,e)}(z,0)
\end{equation*}
Therefore, the map
\begin{equation*}
  \Omega^0(\g_E)\to\Omega^0(\g_E)\qquad
  z\mapsto\chi_{(s,e)}(z,0)=\frac{\sin\ad s}{\ad s}z
\end{equation*}
is invertible.
\begin{proof}[Proof of Proposition~\ref{sec:appendix-existence-Psi}]
  If $(z,w)\in T_{(s,e)}P$, then
  Proposition~\ref{sec:appendix-formula-lambda-chi} implies that
  \begin{equation*}
    (\bar{\kappa},\gamma)_{L^2}(z,w)=(s,\frac{1-\cos\ad s}{\ad s}z+w)_{L^2}
  \end{equation*}
  Since $(s,[s,z])_{L^2}=0$,
  \begin{equation*}
    (s,\frac{1-\cos\ad s}{\ad
      s}z)_{L^2}=\sum_{j=1}^\infty\frac{(-1)^{j-1}}{(2j)!}\Bigl(s,(\ad s)^{2j-1}z\Bigr)_{L^2}=0
  \end{equation*}
  Similarly,
  \begin{align*}
    (\bar{\kappa},\chi)_{L^2}(z,w)
    &=\Bigl(s,\frac{\sin\ad s}{\ad s}z\Bigr)_{L^2}\\
    &=\sum_{j=0}^\infty\frac{(-1)^j}{(2j+1)!}(s,(\ad s)^{2j}
      z)_{L^2}\\
    &=(s,z)_{L^2}
  \end{align*}
  Therefore, the identities $(1)$ and $(2)$ hold at $(s,e)$. By the
  right $\G$-equivariance, they hold everywhere.
  
  Then, we prove the formula for the covariant derivative
  $d_\gamma\bar{\kappa}_{(s,e)}$. Note that
  $d_\gamma\bar{\kappa}=d\bar{\kappa}+[\gamma,\bar{\kappa}]$. Therefore,
  if $(z,w)\in T_{(s,e)}P$, we have
  \begin{align*}
    d_\gamma\bar{\kappa}_{(s,e)}(z,w)
    &=d\bar{\kappa}_{(s,e)}(z,w)+\Bigl[\frac{1-\cos\ad s}{\ad
      s}z+w,s\Bigr]\\
    &=\ad s(w)+z+\ad s\Bigl(\frac{\cos\ad s-1}{\ad
      s}z-w\Bigr)\\
    &=\cos\ad s(z)
  \end{align*}
  To define $\Psi$, it is enough to define $\Psi_{(s,e)}$ which needs
  to satisfy
  \begin{equation*}
    \cos\ad s(z)=\Psi_{(s,e)}\frac{\sin\ad s}{\ad s}(z)
  \end{equation*}
  As a consequence,
  \begin{equation*}
    \Psi_{(s,e)}=\cos\ad s\frac{\ad s}{\sin\ad s}
  \end{equation*}
\end{proof}

The following results verify that $\omega$ is compatible with
$J$. Then, we will verify that $\omega$ is positive so that $\omega$
is a K\"ahler form on $\Omega^0(\g_E)\times\G$.
\begin{proposition}\ The following hold for any
  $\xi,\eta\in\Omega^0(\g_E)$:
  \begin{enumerate}
  \item
    $\omega(\xi^\#,J\eta^\#)=(\xi,\Psi(\eta))_{L^2}$.
   
  \item $\omega(J\xi^\#,J\eta^\#)=\omega(\xi^\#,\eta^\#)$.
    
  \item $(\xi,\Psi(\eta))_{L^2}=(\eta,\Psi(\xi))_{L^2}$.
  \end{enumerate}
  As a consequence, $\omega(J\cdot,J\cdot)=\omega(\cdot,\cdot)$.
\end{proposition}
\begin{proof}
  We compute
  \begin{align*}
    \omega(\xi^\#,J\eta^\#)
    &=-d(\bar{\kappa},\lambda)_{L^2}(\xi^\#,J\eta^\#)\\
    &=-\xi^\#(\bar{\kappa},\lambda(J\eta^\#))_{L^2}+J\eta^\#(\bar{\kappa},\lambda(\xi^\#))+(\bar{\kappa},\lambda([\xi^\#,J\eta^\#])_{L^2}\\
    &=(\bar{\kappa}(J\eta^\#),\xi)_{L^2})+(\bar{\kappa},\lambda(J[\xi^\#,\eta^\#]))_{L^2}\\
    &=(d_\gamma\bar{\kappa}(J\eta^\#),\xi)_{L^2}\\
    &=(\Psi(\eta),\xi)_{L^2}
  \end{align*}
  Here, we have used the formula that
  $[\xi^\#,J\eta^\#]=J[\xi^\#,\eta^\#]$, since $J$ commutes with the
  right $\G$-action. Moreover,
  $d\bar{\kappa}(J\eta^\#)=d_\gamma\bar{\kappa}(J\eta^\#)$, since
  $J\eta^\#$ is horizontal. This proves $(1)$. To prove $(2)$, we
  compute
  \begin{align*}
    \omega(J\xi^\#,J\eta^\#)
    &=(\bar{\kappa},\gamma)_{L^2}([J\xi^\#,J\eta^\#])_{L^2}\\
    &=(\bar{\kappa},\gamma)_{L^2}(-[\xi^\#,\eta^\#])_{L^2}\\
    &=(\bar{\kappa},\gamma)_{L^2}(-[\xi,\eta]^\#)_{L^2}\\
    &=-(\bar{\kappa},[\xi,\eta])_{L^2}
  \end{align*}
  On the other hand,
  \begin{align*}
    \omega(\xi^\#,\eta^\#)
    &=-\xi^\#(\bar{\kappa},\gamma(\eta^\#))_{L^2}+\eta^\#(\bar{\kappa},\gamma(\xi^\#))_{L^2}+(\bar{\kappa},\gamma([\xi^\#,\eta^\#])\\
    &=-([\bar{\kappa},\xi],\eta)_{L^2}+([\bar{\kappa},\eta],\xi)_{L^2}+(\bar{\kappa},[\xi,\eta])_{L^2}\\
    &=(\bar{\kappa},[\eta,\xi])_{L^2}
  \end{align*}
  Finally, to prove $(3)$, we first compute
  \begin{align*}
    d(\bar{\kappa},\chi)_{L^2}(J\xi^\#,J\eta^\#)
    &=J\xi^\#(\bar{\kappa},\chi(J\eta^\#))_{L^2}-J\eta^\#(\bar{\kappa},\chi(J\xi^\#))_{L^2}-(\bar{\kappa},\chi([J\xi^\#,J\eta^\#])_{L^2}\\
    &=(d\bar{\kappa}(J\xi^\#),\eta)_{L^2}-(d\bar{\kappa}(J\eta^\#),\xi)_{L^2}\\
    &=(d_\gamma\bar{\kappa}(J\xi^\#),\eta)_{L^2}-(d_\gamma\bar{\kappa}(J\eta^\#),\xi)_{L^2}\\
    &=(\Psi\chi(J\xi^\#),\eta)_{L^2}-(\Psi\chi(J\eta^\#),\xi)_{L^2}\\
    &=(\Psi(\xi),\eta)_{L^2}-(\Psi(\eta),\xi)_{L^2}
  \end{align*}
  and then note that $d(\bar{\kappa},\chi)_{L^2}=0$ by
  Proposition~\ref{sec:appendix-existence-Psi}.
\end{proof}

\begin{proposition}
  The metric $g=\omega(\cdot,J\cdot)$ is positive-definite. In
  particular, $\omega$ is non-degenerate.
\end{proposition}
\begin{proof}
  If $\xi,\eta\in\Omega^0(\g_E)$, then
  \begin{align*}
    g(\xi^\#+J\eta^\#,\xi^\#+J\eta^\#)
    &=\omega(\xi^\#+J\eta^\#,J\xi^\#-\eta^\#)\\
    &=\omega(\xi^\#,J\xi^\#)-\omega(\xi^\#,\eta^\#)+\omega(J\eta^\#,J\xi^\#)-\omega(J\eta^\#,\eta^\#)\\
    &=(\xi,\Psi(\xi))_{L^2}+(\eta,\Psi(\eta))_{L^2}-2(\bar{\kappa},[\xi,\eta])_{L^2}
  \end{align*}
  Since $g$ is right $\G$-invariant, it is enough to show the
  positive-definiteness at $(s,e)$. Hence, we need to show that
  \begin{align*}
    &(\xi,\Psi_{(s,e)}(\xi))_{L^2}+(\eta,\Psi_{(s,e)}(\eta))_{L^2}-2(s,[\xi,\eta])_{L^2}\\
    &=\int_X\langle\xi,\Psi_{(s,e)}(\xi)\rangle_{L^2}+\langle\eta,\Psi_{(s,e)}(\eta)\rangle_{L^2}-2\langle
    s,[\xi,\eta]\rangle_{L^2}>0
  \end{align*}
  for every $s\in\Omega^0(\g_E)$ and nonzero
  $\xi+i\eta\in\Omega^0(\g_E^\C)$. Note that the integrand is positive
  pointwise. In fact, after fixing a base point $x\in X$, the polar
  decomposition $\psi$ becomes the usual polar decomposition
  $\psi_x\colon\fu(n)\times U(n)\to GL_n(\C)$. The unique
  complex structure on $\fu(n)\times U(n)$ making $\psi_x$ a
  biholomorphism is compatible with the tautological 1-form on
  $\fu\times U(n)=T^*U(n)$ so that the tautological 1-form is
  a K\"ahler form (see \cite[Theorem 5.1 and Remark
  5.2]{Huebschmann2013}). The resulting K\"ahler metric on
  $\fu(n)\times U(n)$ evaluated at $(\xi^\#+J\eta^\#)(x)$
  is exactly the integrand and hence positive. Here,
  $(\xi^\#+J\eta^\#)(x)$ is the value of the section
  $\xi^\#+J\eta^\#\in\Omega^0(\g_E)\oplus\Omega^0(\g_E)$ at $x$. 
\end{proof}

\bibliography{..//..//references}
\bibliographystyle{abbrv}
\end{document}